\documentclass[12pt]{amsart}
\usepackage{amsmath, amscd}
\usepackage{amsthm}
\usepackage{verbatim}
\usepackage{amssymb}
\usepackage[all, cmtip]{xy}
\usepackage[colorlinks=true,pagebackref]{hyperref}
\usepackage[dvipsnames,usenames]{color}
\hypersetup{linkcolor=RawSienna,anchorcolor=BurntOrange,citecolor=OliveGreen,filecolor=BlueViolet,menucolor=Yellow,urlcolor=OliveGreen}

\newtheorem{lemma}{Lemma}[section]
\newtheorem{theorem}[lemma]{Theorem}
\newtheorem{corollary}[lemma]{Corollary}
\newtheorem{proposition}[lemma]{Proposition}

\theoremstyle{definition}
\newtheorem{definition}[lemma]{Definition}
\newtheorem{remark}[lemma]{Remark}

\newtheorem{example}[lemma]{Example}

\newtheorem{exercise}[lemma]{Exercise}

\theoremstyle{remark}
\newtheorem*{remark*}{Remark}
\newtheorem*{note*}{Note}

\newcommand{\Chow}{\operatorname{Ch}}

\newcommand{\GL}{\operatorname{GL}}
\newcommand{\Supp}{\operatorname{Supp}}

\newcommand{\Spec}{\operatorname{Spec}}

\newcommand{\pt}{\operatorname{pt}}

\newcommand{\euler}{\chi}
\newcommand{\CC}{\mathbb{C}}
\newcommand{\ix}{\mathcal{X}}
\newcommand{\iy}{\mathcal{Y}}

\newcommand{\ZZ}{\mathbb{Z}}

\newcommand{\lieg}{\mathfrak{g}}
\newcommand{\QQ}{\mathbb{Q}}

\newcommand{\Mbar}{{\overline{\mathcal M }}}
\newcommand{\A}{\mathbb{A}}

\newcommand{\Pro}{\mathbb{P}}

\newcommand{\ch}{\operatorname{ch}}
\newcommand{\rk}{\operatorname{rk}}
\newcommand{\Td}{\operatorname{Td}}
\newcommand{\td}{\operatorname{td}}

\newcounter{item-counter}

\begin{document}
\title{Riemann-Roch for Deligne-Mumford stacks}

\author{Dan Edidin}
\address{Department of Mathematics, University of Missouri-Columbia, Columbia, Missouri 65211}
\email{edidind@missouri.edu}

\date{October 1, 2012.}

\begin{abstract}
  We give a simple proof of the Riemann-Roch theorem for
  Deligne-Mumford stacks using the equivariant Riemann-Roch theorem
  and the localization theorem in equivariant $K$-theory, together with some
  basic commutative algebra of Artin local rings.
\end{abstract}

\maketitle

%\tableofcontents

\section{Introduction}
The Riemann-Roch theorem is one of the most important and deep
results in mathematics.  At its essence, the theorem gives a method to 
compute the dimension of the space of sections of a vector bundle on
a compact analytic manifold in terms of topological invariants (Chern
classes) of the bundle and manifold. 

Beginning with Riemann's inequality for linear systems on curves, work
on the Riemann-Roch problem spurred the development of fundamental
ideas in many branches of mathematics. In algebraic
geometry Grothendieck viewed the classical Riemann-Roch theorem as
an example  of a transformation between $K$-theory and Chow groups of a
smooth projective variety.  In differential geometry Atiyah and
Singer saw the classical theorem as a special case of their celebrated
index theorem which computes the index of an elliptic operator on a
compact manifold in terms of topological invariants.

Recent work in moduli theory has employed the Riemann-Roch theorem
on Deligne-Mumford stacks. A version of the
theorem for complex $V$-manifolds was proved by Kawasaki
\cite{Kaw:79} using index-theoretic methods. Toen \cite{Toe:99} 
also proved a version of Grothendieck-Riemann-Roch on Deligne-Mumford stacks
using cohomology theories with coefficients
in representations.
Unfortunately, both the statements and proofs that appear
in the literature are quite technical and
as a result somewhat inaccessible to many working in the field.  

The purpose of
this article is to state and prove a version of the Riemann-Roch
theorem for Deligne-Mumford stacks based on the equivariant 
Riemann-Roch theorem for schemes and the localization theorem in
equivariant $K$-theory. Our motivation is the belief that equivariant
methods give the simplest and least technical proof of the
theorem. The proof here is based on the author's joint work with
W. Graham \cite{EdGr:00, EdGr:03, EdGr:05} in equivariant intersection
theory and equivariant $K$-theory. It requires little more background
than some familiarity with Fulton's intersection theory \cite{Ful:84}
and its equivariant analogue developed in \cite{EdGr:98}.

The contents of this article are as follows. In Section \ref{sec.rrforschemes}
we review the algebraic development
of the Riemann-Roch theorem from its original statement for curves to
the version for arbitrary schemes proved by Baum, Fulton and
MacPherson. Our main reference for this materia, with
some slight notational changes, is Fulton's intersection theory book
\cite{Ful:84}. 

In Section \ref{sec.repgrothrr} we explain how
the equivariant Riemann-Roch theorem  \cite{EdGr:00}
easily yields a Grothendieck-Riemann-Roch theorem for representable
morphisms of smooth Deligne-Mumford stacks.

Section \ref{sec.hzrrdm} is the heart of the article. In it we prove
the Hirzebruch-Riemann-Roch theorem for smooth, complete Deligne-Mumford
stacks.  Using the example of the weighted projective line stack
$\Pro(1,2)$ as motivation, we first prove (Section \ref{sec.diaghzrr})
the result for quotient stacks of the form $[X/G]$ with $G$ diagonalizable. This
proof combines the equivariant Riemann-Roch theorem with the classical
localization theorem in equivariant $K$-theory and originally appeared
in \cite{EdGr:03}. In Section \ref{sec.hzrrdmarb} we explain how the
non-abelian localization theorem of \cite{EdGr:05} is used to obtain
the general result. We also include several computations to illustrate how
the theorem can be applied.

In Section \ref{sec.dmgrr} we briefly discuss the
Grothendieck-Riemann-Roch theorem for Deligne-Mumford stacks and
illustrate its use by computing the Todd class of a weighted
projective space.

For the convenience of the reader we also include an Appendix with some basic definitions used in the theory.

{\bf Acknowledgment:} The author is grateful to the referee for a very 
thorough reading of an earlier version of this article.

{\bf Dedication:} It is a pleasure to dedicate this article to my teacher, Joe Harris.

\section{The Riemann-Roch theorem for
  schemes} \label{sec.rrforschemes} The material in Sections
\ref{sec.hirzrr} - \ref{sec.fulrr} is well known and further details 
can be found in
the book \cite{Ful:84}.

\subsection{Riemann-Roch through Hirzebruch} \label{sec.hirzrr}
The original Riemann-Roch theorem is a statement about curves. If $D$ is a divisor on a smooth complete curve $C$ then the result can be stated as:
$$l(D) - l(K_C -D) = \deg D + 1 -g$$
where $K_C$ is the canonical divisor and $l(D)$ indicates the dimension of the linear series of effective divisors equivalent to $D$. 
Using Serre duality we can rewrite this as
$$\euler(C,L(D)) = \deg D + 1-g.$$
where $L(D)$ is the line bundle determined by $D$.
Or, in slightly fancier notation
\begin{equation} \label{eq.rrforcurves}
\euler(C,L(D)) = \deg c_1(L(D)) + 1-g.
\end{equation}

The Hirzebruch-Riemann-Roch theorem extends \eqref{eq.rrforcurves} to arbitrary 
smooth complete varieties.
\begin{theorem}[Hirzebruch-Riemann-Roch]
Let $X$ be a smooth projective variety and let $V$ be a vector bundle on $X$.
Then 
\begin{equation}
\euler(X,V) = \int_X \ch(E) \Td(X)
\end{equation}
where $\ch(V)$ is the Chern character of $V$, $\Td(X)$ is the Todd class of the tangent bundle and $\int_X$ is refers to the degree of the 0-dimensional component in the product.
\end{theorem}
The Hirzebruch version of Riemann-Roch yields many useful
formulas. For example,  if $X$ is a smooth algebraic surface then the arithmetic
genus can be computed as 
\begin{equation} \label{eq.chisurface}
\euler(X,{\mathcal O}_X) = {1\over{12}} \int_X c_1^2 + c_2 = {1\over{12}}(K^2 + \chi)
\end{equation}
where $\chi$ is the topological genus.

\begin{comment}Perhaps even more importantly, Hirzebruch-Riemann-Roch is the ancestor
of some of the fundamental theorems in mathematics. In the analytic
direction, the Atiyah-Singer index theorem generalizes
\eqref{eq.chisurface} to an analogous formula for the index of a
differential operator on *get details*, and on the algebraic side Grothendieck proved a relative version, now called the Grothendieck-Riemann-Roch theorem.
\end{comment}
\subsection{The Grothendieck-Riemann-Roch theorem} \label{sec.grr}
The Grothendieck-Riemann-Roch theorem extends the Hirzebruch-Riemann-Roch
theorem to the relative setting. Rather than considering Euler
characteristics of vector bundles on smooth, complete varieties we consider the 
relative Euler characteristic for proper morphisms of smooth varieties.

Let $f \colon X \to Y$ be a proper morphism of smooth varieties. The Chern
character defines homomorphisms $\ch \colon K_0(X) \to \Chow^*X \otimes \QQ$, 
and $\ch \colon
K_0(Y) \to \Chow^*Y \otimes \QQ$. Likewise, there are two pushforward maps: the relative
Euler characteristic $f_* \colon K_0(X) \to K_0(Y)$ and proper pushforward
$f_* \colon \Chow^*(X) \to \Chow^*(Y)$.  Since we have 4 groups and 4 natural maps we obtain a diagram - which
which does not commute!
\begin{equation} \label{diag.nocommute}
\begin{array}{ccc}
K_0(X) & \stackrel{\ch} \to & \Chow^*(X) \otimes \QQ\\
f_* \downarrow & & f_* \downarrow \\
K_0(Y) & \stackrel{\ch} \to & \Chow^*(Y) \otimes \QQ
\end{array}
\end{equation}
The Grothendieck-Riemann-Roch theorem supplies the correction that makes
\eqref{diag.nocommute} commutative.  If $\alpha \in K_0(X)$
then
\begin{equation} \label{eq.grr}
\ch(f_*\alpha) \Td(Y) = f_*\left(\ch(\alpha) \Td(X) \right)\in 
\Chow^*(Y) \otimes \QQ.
\end{equation}
In other words the following diagram commutes:
\begin{equation} \label{diag.commute}
\begin{array}{ccc}
K_0(X) & \stackrel{\ch \Td(X)} \to & \Chow^*(X) \otimes \QQ\\
f_* \downarrow & & f_* \downarrow \\
K_0(Y) & \stackrel{\ch \Td(Y)} \to & \Chow^*(Y) \otimes \QQ
\end{array}
\end{equation}
Since $\Td(Y)$ is invertible in $\Chow^*(Y)$ we can rewrite equation
\eqref{eq.grr} as
\begin{equation} \label{eq.grr2}
\ch(f_*\alpha) = f_*\left(\ch(\alpha) \Td(T_f)\right)
\end{equation}
where $T_f= [TX] -[f^*TY] \in K_0(X)$ is the relative tangent bundle.

\begin{example}
Equation \eqref{eq.grr2} can be viewed as a relative version of the
Hirzebruch-Riemann-Roch formula, but it is also  more
general. For example, it can also be applied when $f \colon X \to Y$ is a regular
embedding of codimension $d$. 
In this case a more refined statement holds. If $N$ is the normal
bundle of $f$ and $V$ is a vector bundle of rank $r$ on $X$ then the equation
$$c(f_*V)) = 1 + f_*P(c_1(V), \ldots , c_r(V), c_1(N), \ldots c_d(N))$$
holds in $\Chow^*(Y)$
where $P(T_1, \ldots , T_d, U_1, \ldots , U_d)$ is a universal power series
with integer coefficients.

This result is known as Riemann-Roch without denominators and was conjectured
by Grothendieck and proved by Grothendieck and Jouanolou.
\end{example}

\subsection{Riemann-Roch for singular schemes} \label{sec.fulrr}
If $Z \subset X$ is a subvariety of codimension $k$ then
$\ch[{\mathcal O}_Z] = [Z] + \beta$ where $\beta$ is an element of
$\Chow^*(X)$ supported in codimension strictly greater than $k$.
Since $\Td(X)$ is invertible in $\Chow^*(X)$ the 
Grothendieck-Riemann-Roch theorem can be restated as follows:
\begin{theorem} \label{thm.grr2} The map $\tau_X \colon K_0(X) \to
  \Chow^*(X)\otimes \QQ$ defined by $[V] \mapsto \ch(V) \Td(X)$ 
is covariant for proper morphisms of
  smooth schemes\footnote{This means that if $f\colon X \to Y$ is a
    proper morphism of smooth schemes then $f_* \circ \tau_X = \tau_Y
    \circ f_*$ as maps $K_0(X) \to \Chow^*(Y) \otimes \QQ$.} and becomes an
  isomorphism after tensoring $K_0(X)$ with $\QQ$.
\end{theorem}
The Riemann-Roch theorem of Baum, Fulton and MacPherson generalizes
previous Riemmann-Roch theorems to maps of arbitrary schemes. However,
the Grothendieck group of vector bundles $K_0(X)$ is replaced by the
Grothendieck group of coherent sheaves $G_0(X)$.
\begin{theorem} \label{thm.fultrr} \cite[Theorem 18.3, Corollary 18.3.2]{Ful:84}
  For all schemes $X$ there is a homomorphism $\tau_X \colon G_0(X) \to
  \Chow^*(X) \otimes \QQ$ satisfying the following properties:

(a) $\tau_X$ is covariant for proper morphisms.

(b) If $V$ is a vector bundle on $X$ then $\tau_X([V]) = \ch(V) \tau_X({\mathcal O}_X)$.

(c) If $f \colon X \to Y$ 
is an lci morphism with relative tangent bundle $T_f$ then for every class
$\alpha \in G_0(Y)$
$\tau_X f^*\alpha = \Td(T_f) \cap f^*\tau(\alpha)$.

(d) If $Z \subset X$ is a  subvariety of codimension $k$
then $\tau({\mathcal O}_Z) = [Z] + \beta$
where $\beta \in \Chow^*(X)$ is supported in codimension strictly greater than $k$.

(e) The map $\tau_X$ induces an isomorphism $G_0(X) \otimes \QQ \to \Chow^*(X) \otimes \QQ$.
\end{theorem}
\begin{remark} 
  If $X$ is smooth then $K_0(X) = G_0(X)$ and using (c) we see that
  $\tau_X({\mathcal O}_X) = \Td(X)$ and thereby obtain the Hirzebruch and
  Grothendieck Riemann-Roch theorems.  In \cite{Ful:84} the Chow class
  $\tau_X({\mathcal O}_X)$ is called the {\em Todd class} of $X$.
\end{remark}
\begin{remark}
Theorem \ref{thm.fultrr} is proved by a reduction to the (quasi)-projective case
via Chow's lemma. Since Chow's lemma also holds for algebraic spaces, the 
Theorem immediately extends to the category of algebraic spaces.
\end{remark}

\section{Grothendieck Riemann-Roch for representable morphisms of quotient
Deligne-Mumford stacks} \label{sec.repgrothrr}
The goal of this section explain how the equivariant
Riemann-Roch theorem \ref{thm.equivariantrr} yields a
Grothendieck-Riemann-Roch theorem for {\em representable} morphisms of
Deligne-Mumford quotient stacks. 

\subsection{Equivariant Riemann-Roch} \label{sec.equivrr} If $G$ is an
algebraic group acting on a scheme $X$ then there are equivariant
versions of $K$-theory, Chow groups and Chern classes (see the
appendix for definitions). Thus it is natural to expect an equivariant
Riemann-Roch theorem relating equivariant $K$-theory with equivariant
Chow groups. Such a theorem was proved in \cite{EdGr:00} for the
arbitrary action of an algebraic group $G$ on a separated algebraic
space $X$. Before we state the equivariant Riemann-Roch theorem we
introduce some further notation.

The equivariant Grothendieck group of coherent sheaves, $G_0(G,X)$, is a
module for both $K_0(G,X)$, the Grothendieck ring of $G$-equivariant
vector bundles, and $R(G) = K_0(G,\pt)$, the Grothendieck ring of
$G$-modules. Each of these rings has a distinguished ideal, the
augmentation ideal, corresponding to virtual vector bundles
(resp. representations) of rank 0. A result of \cite{EdGr:00} shows
that the two augmentation ideals generate the same topology on
$G_0(G,X)$ and we denote by $\widehat{G_0(G,X)}$ the completion of
$G_0(G,X)_\QQ$ with respect to this topology.

The equivariant Riemann-Roch theorem generalizes Theorem \ref{thm.fultrr}
as follows:
\begin{theorem}\label{thm.equivariantrr}
There is a homomorphism $\tau_X \colon G_0(G,X) \to \prod_{i=0}^\infty \Chow^i_G(X)\otimes \QQ$ which factors through an isomorphism $\widehat{G_0(G,X)} \to \prod_{i=0}^\infty \Chow^i_G(X)\otimes \QQ$. The map $\tau_X$ is covariant for proper equivariant  morphisms
and when $X$ is a smooth scheme and $V$ is a vector bundle then
\begin{equation} \label{eq.equivchtd}
\tau_X(V) = \ch(V)\Td(TX -\lieg)
\end{equation} 
where $\lieg$ is the adjoint representation of $G$.
 \end{theorem}

\begin{remark} The $K$-theory class $TX - \lieg$ appearing \eqref{eq.equivchtd}
corresponds to the tangent bundle of the quotient stack $[X/G]$. If $G$ is finite
then $\lieg = 0$ and if $G$ is diagonalizable (or more generally solvable)
then $\lieg$ is a trivial representation of $G$ and the formula
$\tau_X(V)= \ch(V) \Td(TX)$ also holds.
\end{remark}

\begin{example}
If $X = \pt$ and $G = \CC^*$ then $R(G)$ is the representation ring of
$G$. Since $G$ is diagonalizable the representation ring is generated by
characters and $R(G) = \ZZ[\xi,\xi^{-1}]$ where $\xi$ is the character of weight one. If we set $t = c_1(\xi)$ then the map $\tau_X$ is simply the exponential
map $\ZZ[\xi,\xi^{-1}] \to \QQ[[t]]$, $\xi \mapsto e^t$.
The augmentation ideal of $R(G)$ is ${\mathfrak m} = (\xi -1)$.
If we tensor with $\QQ$ and complete at the ideal ${\mathfrak m}$
then the completed ring $\widehat{R(G)}$ is isomorphic to the power series
ring $\QQ[[x]]$ where $x = \xi -1$.  The map $\tau_X$ is the isomorphism
sending $x$ to $e^t-1= t(1 + t/2 + t^2/3! + \ldots )$.
\end{example}

\subsection{Quotient stacks and moduli spaces}
\begin{definition}
  A quotient stack is a stack ${\mathcal X}$ equivalent to the
  quotient $[X/G]$ where $G \subset \GL_n$ is a linear algebraic group
  and $X$ is a scheme (or more generally an algebraic space\footnote{The
    fact that $X$ is an algebraic space as opposed to a scheme makes
    little difference in this theory.}).

A quotient stack is Deligne-Mumford if the stabilizer of every point
is finite and geometrically reduced. Note that in characteristic 0 the second
condition is automatic.

A quotient stack $\ix = [X/G]$ is {\em separated} if the action of $G$
on $X$ is proper - that is, the map $\sigma \colon G \times X \to X
\times X$, $(g,x) \mapsto (gx,x)$ is proper. Since $G$ is affine
$\sigma$ is proper if and only if it is finite. In characteristic 0 any separated quotient stack is automatically a Deligne-Mumford stack.
\end{definition}

The hypothesis that a Deligne-Mumford stack is a quotient stack is
not particularly restrictive. Indeed, the author does not know any
example of a separated Deligne-Mumford stack which is not a quotient
stack. Moreover, there are a number of general results which show that
``most'' Deligne-Mumford stacks are quotient stacks
\cite{EHKV:01,KrVi:04}. For example if ${\mathcal X}$ satisfies the
resolution property - that is, every coherent sheaf is the quotient of
a locally free sheaf then ${\mathcal X}$ is quotient stack.

It is important to distinguish two classes of morphisms of
Deligne-Mumford stacks, {\em representable} and {\em
  non-representable} morphisms. Roughly speaking, a morphism of Deligne-Mumford
stacks $\ix \to {\mathcal Y}$ is representable if the fibers of $f$
are schemes. Any morphism $X' \to \ix$ from a scheme to a
Deligne-Mumford stack is representable. If $\ix = [X/G]$ and $\iy =
[Y/H]$ are quotient stacks and $f \colon \ix \to \iy$ is representable
then $\ix$ is equivalent to a quotient $[Z/H]$ (where $Z =
Y\times_{\iy} \ix$ )and the map of stacks $\ix \to \iy$ is induced by
an $H$-equivariant morphism $Z \to Y$. Thus, for quotient stacks we
may think of representable morphisms as those corresponding to
$G$-equivariant morphisms.

The non-representable morphisms that we will encounter are all
morphisms from a Deligne-Mumford stack to a scheme or algebraic
space. Specifically we consider the structure map from a
Deligne-Mumford stack to a point and the map from a stack to its
coarse moduli space.

Every Deligne-Mumford stack ${\mathcal X}$ is {\em finitely
  parametrized}.  This means that there is finite surjective morphism
$X' \to \ix$ where $X$ is a scheme. Thus we can say that a separated
stack $\ix$ is {\em complete} if it is finitely parametrized by a complete scheme.

A deep result of Keel and Mori \cite{KeMo:97} implies that every
separated Deligne-Mumford stack $\ix$ has a {\em coarse moduli} space
$M$ in the category of algebraic spaces. Roughly speaking, this means
that there is a proper surjective (but not representable) morphism $p
\colon \ix \to M$ which is a bijection on geometric points and
satisfies the universal property that any morphism $\ix \to M'$ with
$M'$ an algebraic space must factor through $p$. When $\ix = [X/G]$
then the coarse moduli space $M$ is the geometric quotient in the
category of algebraic spaces. When $X=X^s$ is the set of stable points
for the action of a reductive group $G$ then $M$ is the geometric
invariant theory quotient of \cite{MFK:94}.

The map $\ix \to M$ is not finite in the usual scheme-theoretic sense,
because it is not representable, but it behaves like a finite morphism
in the sense that if $f\colon X' \to \ix$ is a finite parametrization
then the composite morphism $X' \to M$ is finite. Note, however, that
if we define $\deg p$ by requiring $\deg p \deg f =\deg X'/M$ then
$\deg p$ may be fractional (see below).

Since $p$ is a bijection on geometric points, some of the geometry of
the stack $\ix$ can be understood by studying the coarse space $M$.
Note, however, that when $\ix$ is smooth the
space $M$ will in general have finite quotient singularities.

\subsubsection{$K$-theory and Chow groups of quotient stacks}
If $\ix$ is a stack then we use the notation $K_0(\ix)$ to denote the
Grothendieck group of vector bundles on $\ix$ and we denote by
$G_0(\ix)$ the Grothendieck group of coherent sheaves on $\ix$. If $\ix$
is smooth and has the resolution property then the natural map $K_0(\ix)
\to G_0(\ix)$ is an isomorphism. 

If $\ix = [X/G]$ then $K_0(\ix)$ (resp. $G_0(\ix)$) 
is naturally identified with the equivariant Grothendieck ring $K_0(G,X)$
(resp. equivariant Grothendieck group $G_0(G,X)$. 

Chow groups of Deligne-Mumford stacks were defined with rational
coefficients by Gillet \cite{Gil:84} and Vistoli \cite{Vis:89}
and with integral coefficients
by Kresch \cite{Kre:99}.  When $\ix = [X/G]$ Kresch's Chow groups agree
integrally with the equivariant Chow groups $\Chow^*_G(X)$ defined in
\cite{EdGr:98}. The proper pushforward of rational Chow groups $p
\colon \Chow^*(\ix) \otimes \QQ \to \Chow^*(M)\otimes \QQ$ is an always an
isomorphism \cite{Vis:89,EdGr:98}. In particular this means that if $\ix = [X/G]$
is a Deligne-Mumford stack then every equivariant Chow 
class can be represented by a $G$-invariant cycle
on $X$ (as opposed to $X \times {\mathbf V}$ where ${\mathbf V}$ is a representation of $G$).
Consequently $\Chow^k(\ix) \otimes \QQ = 0$ for $k > \dim \ix$.

The theory of Chern classes in equivariant intersection theory implies that
a vector bundle $V$ on $\ix = [X/G]$ has
Chern classes $c_i(V)$ which operate on $\Chow^*(\ix)$. If $\ix$ is smooth
then we may again view the Chern classes as elements of $\Chow^*(\ix)$.
If $\ix$ is smooth and Deligne-Mumford 
the Chern character and Todd class are again maps
$ K_0(\ix) \to \Chow^*(\ix)\otimes \QQ$. 

Every smooth Deligne-Mumford stack has a tangent bundle. If $\ix =
[X/G]$ is a quotient stack then the map $X \to [X/G]$ is a $G$-torsor
so the tangent bundle to $\ix$ corresponds to the quotient $TX/\lieg$
where $\lieg$ is the adjoint representation of $G$.  
In particular under the
identification of $\Chow^*(\ix) = \Chow^*_G(X)$, $c(T\ix) =
c(TX)c(\lieg)^{-1}$. If $G$ is finite or diagonalizable then $\lieg$ is a trivial representation so
$c_t(\lieg) = 1$. Thus, the Chern classes of $T\ix$ are just the
equivariant Chern classes of $TX$ in these cases.

\subsubsection{Restatement of the  equivariant Riemann-Roch theorem for Deligne-Mumford quotient stacks}

As already noted, when $G$ acts properly then $\Chow^i_G(X)_\QQ = 0$
for $i > \dim [X/G]$ so the infinite direct product in Theorem
\ref{thm.equivariantrr} is just $\Chow^*(\ix)$ where $\ix = [X/G]$. A
more subtle fact proved in \cite{EdGr:00} is that if $G$ acts with
finite stabilizers (in particular if the action is proper) then
$G_0(G,X) \otimes \QQ$ is supported at a finite number of points 
of $\Spec (R(G) \otimes \QQ)$. It follows that
$\widehat{G_0(G,X)}$ is the same as the localization of the
$R(G)\otimes \QQ$-module $G_0(G,X)\otimes \QQ$ at the augmentation
ideal in $R(G) \otimes \QQ$. For reasons that will become clear in the
next section we denote this localization by $G_0(G,X)_1$ (or
$K_0(G,X)_1$). Identifying equivariant $K$-theory with the $K$-theory
of the stack $\ix = [X/G]$ we will also write $K_0(\ix)_1$ and
$G_0(\ix)_1$ respectively.  Theorem \ref{thm.equivariantrr} implies
the following result about smooth Deligne-Mumford quotient stacks.

\begin{theorem} \label{thm.representableDMrr}
There is a homomorphism $\tau_X \colon G_0(\ix)  \to \Chow^*(\ix)\otimes \QQ$ 
which factors through an isomorphism $G_0(\ix)_1 \to \Chow^*(\ix)\otimes \QQ$. The map $\tau_X$ is covariant for proper representable  morphisms
and when $\ix$ is a smooth and $V$ is a vector bundle then
\begin{equation} \label{eq.equivchtd}
\tau_X(V) = \ch(V)\Td(\ix)
\end{equation} 
\end{theorem}

\section{Hirzebruch Riemann-Roch for quotient Deligne-Mumford
stacks} \label{sec.hzrrdm}
At first glance, Theorem \ref{thm.representableDMrr} looks like the
end of the Riemann-Roch story for Deligne-Mumford stacks, since it
gives a stack-theoretic version of the Grothendieck-Riemann-Roch
theorem for representable morphisms and also explains the relationship
between $K$-theory and Chow groups of a quotient stack. Unfortunately,
the theorem cannot be directly used to compute the Euler characteristic of
vector bundles or coherent sheaves on complete Deligne-Mumford stacks.

The problem is that the Euler characteristic of a vector bundle $V$ on
$\ix$ is the $K$-theoretic direct image $f_!V := \sum
(-1)^iH^i(\ix,V)$ under the projection map $f \colon K_0(\ix) \to
K_0(\pt)=\ZZ$. However, the projection map $\ix \to \pt$ is not
representable - since if it were then $\ix$ would be a scheme or
algebraic space.

A Hirzebruch-Riemann-Roch theorem for a smooth, complete,
Deligne-Mumford stack $\ix$ should be a formula for the Euler characteristic
of a bundle in terms of degrees of Chern characters and Todd classes.
In this section, which is the heart of the paper, we show how to use
Theorem \ref{thm.representableDMrr} and 
generalizations of the localization
theorem in equivariant $K$-theory to obtain such a
result.  Henceforth, we will  work exclusively 
over the complex
numbers $\CC$.

\subsection{Euler characteristics and degrees of 0-cycles}
If $V$ is a coherent sheaf on  $\ix = [X/G]$ then the cohomology groups of
$V$ are representations of $G$
and we make the following definition.
\begin{definition} \label{def.dmeuler}
If $V$ is a $G$-equivariant vector bundle  on $X$ then
Euler characteristic of $V$ viewed as a bundle on ${\mathcal X} =[X/G]$
is $\sum_i (-1)^i \dim H^i(X,V)^G$
where $H^i(X,V)^G$ denotes the invariant subspace. We denote this by $\euler(\ix,V)$. 
\end{definition}

Note that, if $\dim G > 0$ then $X$ will never be complete, so
$H^i(X,V)$ need not be finite dimensional. Nevertheless, if $\ix$ is
complete then $H^i(X,V)^G$ is finite dimensional as it can be
identified with the cohomology of the coherent sheaf $H^i(M,p_*E)$
under the proper morphism $p \colon {\mathcal X} \to M$ from $\ix$ to
its coarse moduli space.

If $G$ is linearly reductive (for example if $G$ is diagonalizable)
then the cohomology group $H^i(X,V)$ decomposes as direct sum of
$G$-modules and $H^i(X,V)^G$ is the trivial summand. In this case it
easily follows that the assignment $V \mapsto \sum_i(-1)^i \dim H^i(X,V)^G$
defines an Euler characteristic homomorphism $K_0(G,X) \to
\ZZ$. The identification of vector bundles on ${\mathcal X}$ with
$G$-equivariant bundles on $X$ yields an Euler characteristic map
$\chi \colon K_0(\ix)\to \ZZ$. When the action of $G$ is free and
${\mathcal X}$ is represented by a scheme, this is the usual Euler
characteristic.

However, even if $G$ is not reductive but acts properly on $X$ then
the assignment $V \mapsto \sum_i (-1)^i \dim H^i(X,V)^G$
still defines an Euler characteristic map $\chi
\colon K_0(\ix) \to \ZZ$.  This follows from Keel and Mori's description
of the finite map $[X/G] \to M=X/G$ as being \'etale locally in $M$ a
quotient $[V/H] \to V/H$ where $V$ is affine and $H$ is finite (and
hence reductive because we work in characteristic 0).

The above reasoning also applies to $G$-linearized coherent sheaves on $X$ 
and we also obtain an Euler characteristic map $\chi \colon G_0(\ix) \to \ZZ$.
These maps can be extended by linearity to maps $\chi \colon K_0(\ix) \otimes F
\to F$ (resp. $G_0(\ix) \otimes F \to F$) where $F$ is any coefficient ring.

\begin{example}
  If  $G$ is a finite group let $BG = [\pt/G]$ be the classifying
  stack parametrizing algebraic $G$ coverings.  
The identity morphism $pt \to pt$ factors as $pt \to
  BG \to pt$ where the first map is the universal $G$-covering and
  which associates to any scheme $T$ the trivial covering $G \times T
  \to T$. The map $BG \to pt$ is the coarse moduli space map and
  associates to any $G$-torsor $Z \to T$ to the ground scheme $T$.

  The map $\pt \to BG$ is representable and the
  pushforward in map $K_0(\pt) \to K_0(BG)$ is the map $\ZZ \to R(G)$
  which sends the a vector space $V$ to the representation $V \otimes
  \CC[G]$ where $\CC[G]$ 
  is the regular representation of $G$.

  Since the $\CC[G]$ contains a copy of the trivial representation
  with multiplicity one, it follows that, with our definition, the
  composition of pushforwards $\ZZ =K_0(pt) \to R(G) = K_0(BG) \to \ZZ =
  K_0(pt)$ is the identity - as expected.
\end{example}

\subsubsection{The degree of a 0-cycle}
Some care is required in understanding $0$-cycles on a Deligne-Mumford
stack.  The reason is that a closed $0$-dimensional integral substack $\eta$ is not
in general a closed point but rather a {\em gerbe}. That is, it is
isomorphic after \'etale base change to $BG$ for some finite group $G$. Assuming that the ground field is algebraically closed then the 
degree of $[\eta]$ is defined to be $1/|G|$. 

If $\ix = [X/G]$ is a complete Deligne-Mumford 
quotient stack then $0$-dimensional integral substacks
correspond to $G$-orbits of closed points and we can define for a closed point
$x \in X$
$\deg [Gx/G] = 1/|G_x|$ where $G_x$ is the stabilizer of $x$.

\begin{example}
The necessity of dividing by the order the stabilizer can be
seen by again looking at the factorization of the  morphism $\pt \to BG \to \pt$
when $G$ is a finite group. The map $\pt \to BG$ has degree $|G|$ so the map
$BG \to \pt$ must have degree ${1\over{|G|}}$.
\end{example}

\subsection{Hirzebruch Riemann-Roch for quotients by diagonalizable groups}
\label{sec.diaghzrr}
The goal of this section is to understand the Riemann-Roch theorem in
an important special case: separated Deligne-Mumford stacks of the
form ${\mathcal X} = [X/G]$ where $X$ is a smooth variety and $G
\subset (\CC^*)^n$ is a diagonalizable group. We will develop the
theory using a very simple example - the weighted projective line
stack $\Pro(1,2)$.

\subsubsection{Example: The weighted projective line stack
  $\Pro(1,2)$, Part I} \label{sec.p12I} Consider the weighted
projective line stack $\Pro(1,2)$. This stack is defined as the
quotient of $[\A^2 \smallsetminus \{0\}/\CC^*]$ where $\CC^*$ acts
with weights $(1,2)$; i.e., $\lambda(v_0,v_1) = (\lambda v_0,
\lambda^2 v_1)$. Because $X=\A^2 \smallsetminus \{0\}$ is an open set
in a two-dimensional representation, every equivariant vector bundle on $X$ is of the
form $X \times V$ where $V$ is a representation of $\CC^*$. In this
example we consider two line bundles on $\Pro(1,2)$ - the line bundle
$L$ associated to the weight one character $\xi$ of $\CC^*$ and the
line bundle ${\mathcal O}$ associated to the trivial character. 

{\bf Direct calculation of $\chi(\Pro(1,2), {\mathcal O})$ and
  $\chi(\Pro(1,2), L)$:} It is easy to compute $\chi(\Pro(1,2), L)$
and $\chi(\Pro(1,2), {\mathcal O})$ directly.  The coarse moduli space
of $\Pro(1,2)$ is the geometric quotient $(\A^2 \smallsetminus
\{0\})/\CC^*$. Even though $\CC^*$ no longer acts freely the quotient
is still $\Pro^1$ since it has a covering by two affines $\Spec
\CC[x_0^2/x_1]$ and $\Spec \CC[x_1/x_0^2]$, where $x_0$ and $x_1$ are
the coordinate functions on $\A^2$.  The Euler characteristic
pushforward $K_0(\Pro(1,2)) \to K_0(\pt) = \ZZ$ factors through the
proper pushforward $K_0(\Pro(1,2)) \to K_0(\Pro^1)$. Consequently, we
can compute $\chi(\Pro(1,2),L)$ and $\chi(\Pro(1,2), {\mathcal O})$ by
identifying the images of these bundles on $\Pro^1$. A direct
computation using the standard covering of $\A^2 \smallsetminus \{0\}$
by $\CC^*$ invariant affines shows that both $L$ and ${\mathcal O}$
pushforward to the trivial bundle on $\Pro^1$. Hence
$$\chi(\Pro(1,2),L)=
\chi(\Pro(1,2), {\mathcal O}) = 1$$

{\bf An attempt to calculate $\chi(\Pro(1,2), {\mathcal O})$ and
  $\chi(\Pro(1,2), L)$ using Riemann-Roch methods:} Following
Hirzebruch-Riemann-Roch for smooth varieties we might expect to compute $\euler(\Pro(1,2),
L)$ as $\int_{\Pro(1,2)} \ch(L) \Td(\Pro(1,2))$. To do that we will
use the presentation of $\Pro(1,2)$ as a quotient by $\CC^*$.  The
line bundle $L$ corresponds to the pullback to $\A^2$ of the standard
character $\xi$ of $\CC^*$ and the tangent bundle to the stack
$\Pro(1,2)$ fits into a weighted Euler sequence
$$0 \to {\mathbf 1} \to \xi + \xi^2  \to  T\Pro(1,2) \to  0$$ where ${\mathbf 1}$
denotes the trivial character of $\CC^*$ and again $\xi$ is the character of 
$\CC^*$ of weight 1. If we let $t  = c_1(\xi)$ then
$$\ch(L) \Td(\Pro(1,2)) = (1 + t)(1 + 3t/2) = 1+ 5t/2$$
Now the Chow class $t$ is represented by the invariant cycle $[x=0]$ on $\A^2$
and the corresponding point of $\Pro(1,2)$ has stabilizer of order 2.
Thus 
$$\int_{\Pro(1,2)} \ch(L) \Td(\Pro(1,2)) = 1/2(5/2)= 5/4$$
which is 1/4 too big.
On the other a hand 
then again $\euler(\Pro(1,2), {\mathcal O})= 1$ but
$$\int_{\Pro(1,2)} \ch({\mathcal O}) \Td(\Pro(1,2)) = 3/4$$
is too small by 1/4.
In particular
\begin{equation} \label{eq.whodunnit}
 \int_{\Pro(1,2)} \ch({\mathcal O} +L) \Td(\Pro(1,2)) =2
\end{equation}
which is indeed equal to $\chi(\Pro(1,2), {\mathcal O} + L)$.

Equation \eqref{eq.whodunnit} may seem unremarkable but is in fact a hint as to how to obtain a Riemann-Roch formula that works for all bundles on $\Pro(1,2)$.

\subsubsection{The support of equivariant $K$-theory} \label{sec.supp}
To understand why \eqref{eq.whodunnit} holds
we need to study $K_0(\Pro(1,2)$ as an $R(\CC^*)$-module. Precisely,
$$K_0(\Pro(1,2)) = K_0(\CC^*,\A^2 \smallsetminus \{0\}) =
\ZZ[\xi,\xi^{-1}]/(\xi^2-1)(\xi-1).$$ This follows from the fact that $\A^2$ is
a representation of $\CC^*$ so $K_0(\CC^*,\A^2) = R(\CC^*) =
\ZZ[\xi,\xi^{-1}]$ where again $\xi$  denotes the weight one character of
$\CC^*$. Because
we delete the origin we must quotient by the ideal generated by the
$K$-theoretic Euler class of the tangent space to the origin. With our
action, $\A^2$ is the representation $\xi+ \xi^2$ 
so the tangent space of the origin is also $\xi + \xi^2$. The Euler class of this representation is
$(1 - \xi^{-1})(1  - \xi^{-2})$ which 
generates the ideal $(\xi^2 -1)(\xi-1)$.

From the above description we see that $K_0(\CC^*,\A^2\smallsetminus
\{0\}) \otimes \CC$ is an Artin ring supported at the points $1$ and
$-1$ of $\Spec R(G) \otimes \CC = \CC^*$.  The vector bundle
${\mathcal O} + L$ on $\Pro(1,2)$ corresponding to the element $1 +
\xi \in R(\CC^*)$ is supported at $1\in \CC^*$ and the formula
$$\chi(\Pro(1,2), {\mathcal O} + L) = \int_{\Pro(1,2)}(\ch({\mathcal O} + L)\Td(\Pro(1,2))$$ is correct.  On the
other hand the class of the  bundle ${\mathcal O}$
decomposes as $ [{\mathcal O}]_{1} + [{\mathcal O}]_{-1}$ where
$[\mathcal O]_{1}= 1/2(1 + \xi)$ is supported at $1$
and $[\mathcal O]_{-1} = 1/2(1 - \xi)$ is
supported at $-1$.  In this case the integral $\int_{\Pro(1,2)}
\ch({\mathcal O})\Td(\Pro(1,2))$ computes $\euler(\Pro(1,2),[{\mathcal
  O}]_{1})$.

This phenomenon is general. 
If $\alpha \in K_0(G,X)\otimes \QQ$, denote by $\alpha_{1}$ the component supported
at the augmentation ideal of $R(G)$. 
\begin{corollary} \label{cor.firststep}\cite[cf. Proof of Theorem
  6.8]{EdGr:05} Let $G$ be a linear algebraic group (not necessarily diagonalizable) acting properly on smooth variety $X$. Then if
  $\alpha \in K_0(\ix) \otimes \QQ$
\begin{equation} \label{eq.onepiece}
\int_{{\mathcal X}} \ch(\alpha) \Td({\mathcal X}) = \euler(\ix, \alpha_{1}).
\end{equation}
\end{corollary}
\begin{proof}
Since the equivariant Chern character map factors through $K_0(G,X)_1$
it suffices to prove that
\begin{equation} \label{eq.onetwopiece}
\int_{{\mathcal X}} \ch(\alpha) \Td({\mathcal X}) = \euler(\ix, \alpha)
\end{equation}
for $\alpha \in K_0(G,X)_1$.  To prove our result we use the fact that
every Deligne-Mumford stack $\ix$ is finitely
parametrizable. Translated in terms of group actions this means
that there is a finite, surjective $G$-equivariant morphism $X' \to X$
such that $G$ acts freely on $X'$ and the quotient $\ix' = [X'/G]$ is
represented by a scheme. (This result was first proved by Seshadri in
\cite{Ses:72} and is the basis for the finite parametrization theorem
for stacks proved in \cite{EHKV:01}.) The scheme $X'$ is in general
singular\footnote{If the quotient $X/G$ is quasi-projective then a
  result of Kresch and Vistoli \cite{KrVi:04} shows that we can take
  $X'$ to be smooth, but this is not necessary for our purposes.}, but
the equivariant Riemann-Roch theorem implies the following proposition.
\begin{proposition} \label{cor.inthemiddle}
Let $G$ act properly on $X$ and 
let $f \colon X' \to X$ be a finite surjective $G$-equivariant map.
Then 
the proper pushforward  $f_*
\colon G_0(G,X') \to G_0(G,X)$ induces a surjection 
$G_0(G,X')_1
\to G_0(G,X)_1$, where $G_0(G,X)_1$ (resp. $G_0(G,X)_1$) denotes the localization of $G_0(G,X)\otimes \QQ$ (resp. $G_0(G,X')\otimes \QQ$) 
at the augmentation ideal of $R(G) \otimes \QQ$.
\end{proposition}
\begin{proof}[Proof of Proposition \ref{cor.inthemiddle}]
Because $G$ acts properly on $X$ and $X' \to X$ is finite (hence proper) it
follows that $G$ acts properly on $X'$. Thus $\Chow^*_G(X')\otimes \QQ$
and $\Chow^*_G(X) \otimes \QQ$ are generated by $G$-invariant cycles.
Since $f$ is finite and surjective any $G$-invariant cycle on $X$ is
the direct image of some rational $G$-invariant cycle on $X'$; i.e.,
the pushforward of Chow groups $f_* \colon \Chow^*_G(X') \to \Chow^*_G(X)$
is surjective after tensoring with $\QQ$. Hence by Theorem 
\ref{thm.representableDMrr} the corresponding map 
$f_* \colon G_0(G,X')_1 \to G_0(G,X)_1$ is also surjective.
\end{proof}

Now $G$ acts freely on $X'$ so $G_0(G,X') \otimes \QQ$ is supported entirely at
the augmentation ideal of $R(G) \otimes \QQ$. Therefore we have a surjection
$G_0(G,X')\otimes \QQ \to G_0(G,X)_1$. Since $X$ is smooth, we can
also identify $K_0(G,X)_1$ with $G_0(G,X)_1$ and express the class
$\alpha \in K_0(G,X)_1$ as $\alpha = f_* \beta$.  Since $f$ is finite
we see that $\euler(\ix',\alpha) = \euler(\ix, \beta)$. Since $\ix'$
is a scheme, we know by the Riemann-Roch theorem for the singular
schemes that $\euler(\ix', \beta) = \int_{\ix'}\tau_{\ix'}(\beta)$.  Applying
the covariance of the equivariant Riemann-Roch map for proper
equivariant morphisms we conclude that
$$\int_{\ix} \ch(\alpha)\Td(\ix) = \int_{\ix'} \tau_{\ix'}(\beta) = \euler(\ix,\beta) = \euler(\ix,\alpha).$$
\end{proof}

\subsubsection{The localization theorem in equivariant $K$-theory}

Corollary \ref{cor.firststep} tells us how to deal with the component
of $G_0(G,X)$ supported 
at the augmentation ideal. We now turn to the problem of understanding what to do with the rest of equivariant $K$-theory. The key tool is the {\em localization theorem}. 

The correspondence between diagonalizable groups and finitely
generated abelian groups implies that if $G$ is a complex
diagonalizable group then $R(G) \otimes \CC$ is the coordinate ring of $G$.
Since the $R(G) \otimes \QQ$-module $G_0(G,X) \otimes \QQ$ is supported at a
finite number of closed points of $\Spec R(G) \otimes \QQ$ it follows
that $G_0(G,X) \otimes \CC$ is also supported at a finite number of
closed points of $G = \Spec R(G) \otimes \CC$.  If $h \in G$ then we
denote by $G_0(G,X)_h$ the localization of $G_0(G,X)\otimes \CC$ at the
corresponding maximal ideal of $R(G) \otimes \CC$. In the course of 
the proof of \cite[Theorem 2.1]{Tho:92} Thomason showed 
that $G_0(G,X)_h=0$ if $h$ acts without fixed point on
$X$. Hence $h \in \Supp G_0(G,X)$ implies that $X^h \neq
\emptyset$. Since $G$ is assumed to act with finite stabilizers
(because it acts properly) it follows that $h$ must be of finite order
if $h \in \Supp G_0(G,X)$.

If $X$ is a smooth scheme then we can identify $G_0(G,X) = K_0(G,X)$ and the discussion of the above paragraph applies to the Grothendieck ring of vector bundles.

Let $X^h$ be the fixed locus of $h \in
G$. If $X$ is smooth then $X^h$ is a smooth closed subvariety of $X$ so the inclusion
$i_h \colon X^h \to X$ is a regular embedding.  Since the map $i_h$ is
$G$-invariant the normal bundle $N_h$ of $X^h \to X$ comes with a natural $G$-action. The key to understanding what happens to the summand $G_0(G,X)_h$ is the localization theorem:
\begin{theorem}
  Let $G$ be a diagonalizable group acting on a smooth variety $X$.
  The pullback $i_h^* \colon G_0(G,X) \to G_0(G,X^h)$ is an isomorphism
  after tensoring with $\CC$ and localizing at $h$. Moreover, the Euler
  class of the normal bundle, $\lambda_{-1}(N_h^*)$, is invertible
  in $G_0(G,X^h)_h$ and if $\alpha \in G_0(G,X)$ then
$$\alpha = (i_h)_*\left( {i_h^*\alpha \over{\lambda_{-1}(N^*_h)}}\right)$$
\end{theorem}
\begin{remark}
  The localization theorem in equivariant $K$-theory was originally
  proved by Segal in \cite{Seg:68b}. The version stated above is essentially
\cite[Lemma 3.2]{Tho:92}.
\end{remark}

\subsubsection{Hirzebruch-Riemann-Roch for diagonalizable group actions}
The localization theorem implies that if $\alpha \in G_0(G,X)_h$ then
$$\euler(\ix, \alpha) = \euler([X^h/G],\frac{i_h^*\alpha}
{\lambda_{-1}N^*_h}).$$ Thus if $\alpha \in G_0(G,X)_h$ then we can 
compute $\euler([X/G],\alpha)$ by restricting to the fixed locus $X^h$.
This is advantageous because there is an 
automorphism of $G_0(G,X^h)$ which moves the component of a $K$-theory class
supported at $h$ to the component supported at $1$ without changing
the Euler characteristic.

\begin{definition}
  Let $V$ be a $G$-equivariant vector bundle on a space $Y$ and suppose
  that an element $h \in G$ of finite order acts trivially on $Y$. Let $H$ be the cyclic group
generated by $h$ and let $X(H)$ be its character group.
Then $V$ 
decomposes into a sum of $h$-eigenbundles  $\oplus_{\xi\in X(H)}V_\xi$
for the action of $H$ on the fibres of $V \to Y$. Because the action
of $H$ commutes with the action of $G$ (since $G$ is abelian) each
eigenbundle is a $G$-equivariant vector bundle.
Define  $t_h([V]) \in K_0(G,Y) \otimes \CC$ to be the class of the virtual
bundle 
$\sum_{\xi \in X(H)} \xi(h) V_\xi$.  
A similar construction for $G$-linearized coherent sheaves defines an automorphism $t_h \colon G_0(G,Y) \otimes \CC
\to G_0(G,Y) \otimes \CC$.
\end{definition}

The map $t_h$ is compatible with the automorphism of $R(G) \otimes \CC$ 
induced by
the translation map $G \to G$, $k \mapsto kh$ and maps the localization
$K_0(G,Y)_h$ to the localization
$K_0(G,Y)_1$.
The analogous statement
also holds for the corresponding localizations of
$G_0(G,Y)\otimes\CC$.

The crucial property of $t_h$ is that it preserves invariants.
\begin{proposition} \label{prop.thatsit}
If $G$ acts properly on $Y$
and $Y/G$  is complete then 
$\euler([Y/G], \beta) = \euler([Y/G], t_h(\beta))$.
\end{proposition}
\begin{proof}
  Observe that if $V = \oplus_{\xi \in X(H)} V_\xi$ then the invariant
  subbundle $V^G$ is contained in the $H$-weight 0 submodule of
  $V$. Since $t_h(E)$ fixes the 0 weight submodule we see that the
  invariants are preserved.
\end{proof}

Combining the localization theorem with Proposition \ref{prop.thatsit} 
we obtain the Hirzebruch-Riemann-Roch theorem for actions of
diagonalizable groups.

\begin{theorem} \label{thm.diagrr}\cite[cf. Theorem 3.1]{EdGr:03}
Let $G$ be a diagonalizable group acting properly on smooth variety $X$
such that the quotient stack $\ix =[X/G]$ is complete.
Then if $V$ is an equivariant vector bundle on $X$
\begin{equation}
\euler(\ix,V) = \sum_{h \in \Supp K_0(G,X)} \int_{[X^h/G]} 
\ch\left(t_h(\frac{i_h^*V}{\lambda_{-1}(N_h^*)})\right)\Td([X^h/G]).
\end{equation}
\end{theorem}

\subsubsection{Conclusion of the $\Pro(1,2)$ example.}
Since $K_0(\Pro(1,2)) = \ZZ[\xi]/(\xi^2-1)(\xi -1)$, we see that
$K$-theory is additively generated by the class $1, \xi, \xi^2$.
We use Theorem \ref{thm.diagrr} to compute
$\euler(\Pro(1,2),\xi^l)$.
First $$\euler(\Pro(1,2), \xi^l_1) = \int_{\Pro(1,2}\ch(\xi^2)\Td(\Pro(1,2))
= \int_{\Pro(1,2)} (1 + lt)(1 + 3/2t) = \int_{\Pro(1,2)}(l + 3/2)t = {(2l+3)\over{4}}.$$
Now we must calculate the contribution from the component supported at $-1$.
If we let $X = \A^2 \smallsetminus \{0\}$
then $X^{-1}$ is the linear subspace $\{(0,a)| a\neq 0\}$.
Because $\CC^*$ acts with weight 2 on $X^{-1}$ the stack $[X^{-1}/\CC^*]$
is isomorphic to the classifying stack $B\ZZ_2$ and  $K_{\CC^*}(X^{-1}) = \ZZ[\xi]/(\xi^2-1)$
while $\Chow^*_{\CC^*}(X^{-1}) = \ZZ[t]/2t$ where again $t = c_1(\xi)$ and 
$\int_{[X^{-1}/\CC^*]} 1 =1/2$. 
Using our formula we see that 
$$\euler(\Pro(1,2),\xi^l_{-1})  = \int_{[X^{-1}/\CC^*]} \ch\left(\frac{(-1)^l\xi^l}
{1 +\xi^{-1}}\right)\Td([X^{-1}/\CC^*]).$$
Since $c_1(\xi)$ is torsion, the only contribution to the integral on the
0-dimensional stack $[X^{-1}/\CC^*]$ is from the class $1$
and we see that
$\euler(\Pro(1,2), \xi^l_{-1}) = (-1)^l/4$, so we conclude that
$$\euler(\Pro(1,2), \xi^l) = {2l + 3 + (-1)^l\over{4}}.$$ In particular,
$\euler(\Pro(1,2), {\mathcal O}) = \euler(\Pro(1,2),L) = 1$. Note however
that  $\euler(\Pro(1,2), L^2)
= 2$.

\begin{exercise}
  You should be able to work things out for arbitrary weighted
  projective stacks.  The stack $\Pro(4,6)$ is known to be isomorphic
  to the stack of elliptic curve $\Mbar_{1,1}$ and so $K_0(\Mbar_{1,1})
  = \ZZ[\xi]/(\xi^4-1)(\xi^6 -1)$.  Hence $K_0(\Mbar_{1,1})$ is
  supported at $\pm 1, \pm i, \omega, \omega^{-1}, \eta, \eta^{-1}$
  where $\omega = e^{2\pi i /3}$ and $\eta = e^{2\pi i /6}$.  Use Theorem
  \ref{thm.diagrr} to compute $\euler(\Mbar_{1,1},\xi^k)$. This computes
  the dimension of the space of level one weight $k$-modular
  forms. The terms in the sum will be complex numbers but the total
  sum is of course integral.
\end{exercise}

\subsubsection{Example: The quotient stack $[(\Pro^2)^3/\ZZ_3]$} \label{sec.p3z3}
To further illustrate Theorem \ref{thm.diagrr} we consider
Hirzebruch-Riemann-Roch on the quotient stack $\ix =
[(\Pro^2)^3/\ZZ_3]$ where $\ZZ_3$ acts on $(\Pro^2)^3$ by cyclic
permutation. This example will serve as a warm-up for Section \ref{sec.p3s3}
where we consider the stack  $[(\Pro^2)^3/S_3]$.

Our goal is to compute $\euler(\ix, L)$ where $L = {\mathcal O}(m)
\boxtimes {\mathcal O}(m) \boxtimes {\mathcal O}(m)$ viewed as a 
$\ZZ_3$-equivariant line bundle on $(\Pro^2)^3$. To make this
computation we observe that $\Chow^*(\ix) =
\Chow^*_{\ZZ_3}((\Pro^2)^3)$ is generated by $\ZZ_3$ invariant cycles. It 
follows that every element 
$\Chow^*(\ix) \otimes \QQ$ is represented by
a symmetric polynomial (of degree at most 6) in the variables $H_1,H_2,H_3$,
where $H_i$ is the hyperplane class on the $i$-th copy of $\Pro^2$.

As before we have that
\begin{equation} \label{eq.firststep}
\euler(\ix, L_{1}) = \int_{\ix} \ch(L)\Td(\ix).
\end{equation}
Since $\ix \to (\Pro^2)^3$ is a $\ZZ_3$ covering we can identify
$T\ix$ with $T((\Pro^2)^3)$ viewed as $\ZZ_3$-equivariant vector bundle. Using
the standard formula for the Todd class of projective space we can rewrite 
equation \eqref{eq.firststep} as 
\begin{equation} \label{eq.secondstep}
\euler(\ix, L_{1}) = \int_{\ix} \prod_{i=1}^3 (1 + m H_i + m^2 H_i^2/2)(1 + 3H_i/2 + H_i^2).
\end{equation}
The only term  which contributes to the integral on the right-hand side of \eqref{eq.secondstep} is $(H_1H_2H_3)^2$. Now if
$P \in \Pro^3$ is any point then $(H_1H_2H_3)^2$ is represented by the invariant
cycle $[P \times P \times P]$. Since $\ZZ_3$ fixes this cycle we see that
$\int_{\ix} [P \times P \times P] = 1/3$ and conclude
that 
\begin{equation}
\euler(X, L_{1}) = 1/3 \left( \text{coefficient of }(H_1H_2H_3)^2\right).
\end{equation}
Expanding the product in \eqref{eq.secondstep} shows that
\begin{equation}
\euler(\ix, L_{1}) = 1/3\left(1 + 9 m/2 + 33 m^2/4 + 63 m^3/8 + 33 m^4/8 + 9 m^5/8 + m^6/8\right).
\end{equation}

Since $R(\ZZ_3) \otimes \CC = \CC[\xi]/(\xi^3 -1)$ we may identify
$\Spec R(\ZZ_3) \otimes \CC$ as the subgroup $\mu_3 \subset \CC^*$
and compute the contributions to $\euler(\ix, L)$ from the components of $L$
supported at $\omega  = e^{2\pi i/3}$ and $\omega^2$.

For both $\omega$ and $\omega^2$ the fixed locus of the corresponding
element of $\ZZ_3$ is the diagonal $\Delta_{(\Pro^2)^3}
\stackrel{\Delta} \hookrightarrow (\Pro^2)^3$. The group $\ZZ_3$ acts
trivially on the diagonal so $K_{\ZZ_3}(\Delta_{(\Pro^2)^3}) =
K_0(\Pro^2) \otimes R(\ZZ_3)$. Under this identification, the pullback
of the tangent bundle of $(\Pro^2)^3$ is $T\Pro^2 \otimes V$ where $V$
is the regular representation of $\ZZ_3$ corresponding to the action
of $\ZZ_3$ on a 3-dimensional vector space by cyclic permutation. Hence $$\Delta^*(T(\Pro^2)^3)) = T\Pro^2 \otimes {\mathbf 1} + T\Pro^2
\otimes \xi + T\Pro^2 \otimes \xi^2$$ where $\xi$ is the character
of $\ZZ_3$ with weight $\omega = e^{2\pi i/3}$.  The $\ZZ_3$-fixed
component of this $\ZZ_3$ equivariant bundle is the tangent bundle to
fixed locus $\Delta_{(\Pro^2)^3}$ and its complement is the normal
bundle. Thus $T{\Delta_{(\Pro^2)^3}} = T\Pro^2 \otimes {\mathbf 1}$ and
$N_{\Delta} = (T\Pro^2 \otimes \xi) + (T\Pro^2 \otimes
\xi^2)$. Computing the $K$-theoretic Euler characteristic gives:
\begin{eqnarray*}
\lambda_{-1}(N_{\Delta}^*) & = &\lambda_{-1}(T^*\Pro^2 \otimes \xi^2) \lambda_{-1}(T^*\Pro^2 \otimes \xi)\\
& = & (1  - T^*_{\Pro^2} \otimes \xi^2 +  K_{\Pro^2} \otimes \xi)
(1  - T^*_{\Pro^2} \otimes \xi + K_{\Pro^2} \otimes \xi^2).
\end{eqnarray*}
(Here we use the fact that $\xi^* = \xi^{-1} = \xi^2$ in $R(\ZZ_3)$.)
Because the above expression is symmetric in $\xi$ and $\xi^2$, applying the twisting operator for either $\omega$ or $\omega^2$ yields
$$
t(\lambda_{-1}(N_{\Delta}^*))  =  (1  - \omega^2 T^*\Pro^2 \otimes \xi^2 + 
\omega K_{\Pro^2} \otimes \xi)(1 - \omega T^*\Pro^2 \otimes \xi + \omega^2 K_{\Pro^2}).$$
Expanding the product in $K$-theory gives:
\begin{equation} \label{eq.lambda1}
t(\lambda_{-1})(N_{\Delta}^*) = 1 + K_{\Pro^2}^2 +
  (T^*\Pro^2)^2 - (T^*\Pro^2 - K_{\Pro^2} + T^*\Pro^2 K_{\Pro^2})
  \otimes (\omega \xi + \omega^2 \xi).
\end{equation}
Expression \eqref{eq.lambda1} simplifies after taking the Chern character because the Chern
classes of any representation are torsion. Precisely,
$$\ch(t(\lambda_{-1}(N_{\Delta}^*))) = 
9 - 27 H + 99 H^2/2.$$
where $H$ is the hyperplane class on $\Delta_{(\Pro^2)^3}$.
Also note that $\Delta^*L = {\mathcal O}(3m) \otimes {\mathbf 1}$ where
${\mathbf 1}$ denotes the trivial representation of $\ZZ_3$. Hence
$t(\Delta^*L) = \Delta^*L$ and
\begin{equation} \label{eq.z3fixed}
\begin{array}{lcl}
\euler(\ix, L_{\omega}) & = & \int_{[\Delta_{(\Pro^2)^3}/\ZZ_3]} \ch({\mathcal O}(3m) \ch(t(\lambda_{-1}(N_{\Delta}^*)^{-1} \Td(\Pro^2)\\
& = & 1/3\left(\text{ coefficient of }H^2\right)\\
& = & 1/3(1 + 3 m/2 + m^2/2)
\end{array}
\end{equation}
with the same answer for $\euler(\ix, L_{\omega^2})$.
Putting the pieces together we see
that 
\begin{equation}
  \euler(\ix, L) 
  = 1 + 5 m/2 + 37 m^2/12 + 21 m^3/8 + 11 m^4/8 + 
  3 m^5/8 + m^6/24.
\end{equation}
\begin{remark}
Note that we have quick consistency check for our computation - namely
that $\euler(\ix,L)$ is an integer-valued polynomial in $m$. The values 
of $\euler(\ix,L)$ for $m = 0,1,2,3$ are $1,11,76, 340$.
\end{remark}
 
\subsection{Hirzebruch Riemann-Roch for arbitrary quotient
  stacks} \label{sec.hzrrdmarb} We now turn to the general case of
quotient stacks $\ix = [X/G]$ with $X$ smooth and $G$ an arbitrary
linear algebraic group acting properly\footnote{Because we work in
  characteristic 0, the hypothesis that $G$ acts properly implies that
  the stabilizers are linearly reductive since they are finite. In
  addition every linear algebraic group over $\CC$ has a Levy
  decomposition $G = LU$ with $L$ reductive and $U$ unipotent and
  normal. If $G$ acts properly then $U$ necessarily acts freely
  because a complex unipotent group has no non-trivial finite
  subgroups. Thus, if we want,  we can quotient by the free action of $U$
  and reduce to the case that $G$ is reductive.}  on $X$. Again
$G_0(\ix) \otimes \CC$ is a module supported a finite number of closed
points of $\Spec R(G) \otimes \CC$. For a  general group $G$, 
$R(G) \otimes \CC$ is the coordinate ring of the quotient of $G$ by its
conjugation action.  As a result, points of $\Spec R(G) \otimes \CC$
are in bijective correspondence with conjugacy classes of semi-simple
(i.e. diagonalizable) elements in $G$. An element $\alpha \in
G_0(G,X)$ decomposes as $\alpha = \alpha_1 + \alpha_{\Psi_2}+
\ldots + \alpha_{\Psi_r}$ where $\alpha_{\Psi_r}$ is the component
supported at the maximal ideal corresponding to the semi-simple
conjugacy class $\Psi_r \subset G$. Moreover, if a conjugacy class
$\Psi$ is in $\Supp G_0(\ix) \otimes \CC$ then $\Psi$ consists of
elements of finite order.

By Corollary \ref{cor.firststep} if ${\mathcal X} = [X/G]$ is complete then
$\euler(\ix, \alpha_1) = \int_{{\mathcal X}} \ch(\alpha) \Td({\mathcal
  X})$. To understand what happens away from the identity we use a
non-abelian version of the  localization theorem proved in
\cite{EdGr:05}. Before we state the  theorem we need to introduce some
notation.  If $\Psi$ is a semi-simple conjugacy class let $S_\Psi =
\{(g,x) | gx = x, g\in\Psi\}$. The condition the $G$ acts properly on
$X$ implies that $S_{\Psi}$ is empty for all but finitely many $\Psi$ and the elements of these $\Psi$ have finite order. In addition,
if $S_\Psi$ is non-empty then the projection $S_\Psi \to X$ is a finite unramified
morphism. 
\begin{remark}
Note that the map $S_\Psi \to X$ need not be an embedding. For example if
$G =S_3$ acts on $X =\A^3$ by permuting coordinates and $\Psi$ is the conjugacy class of two-cycles, then $S_{\Psi}$ is the disjoint union of the three planes
$x=y$, $y=z$, $x=z$.
\end{remark}

If we fix an element $h \in \Psi$ then the map $G \times X^h \to
S_\Psi$, $(g, x) \mapsto (ghg^{-1}, gx)$ identifies $S_\Psi$ as the
quotient $G \times_{Z} X^h$ where $Z={\mathcal Z}_G(h)$ is the centralizer of the
semi-simple element $h \in G$. In particular $G_0(G, S_\Psi)$ can be
identified with $G_0(Z,X^h)$. The element $h$ is central in $Z$
and if $\beta \in G_0(G,S_\Psi)$ we denote by $\beta_{c_\Psi}$ the
component of $\beta$ supported at $h \in \Spec Z$ under the
identification described above. It is relatively straightforward 
\cite[Lemma 4.6]{EdGr:05} to
show that $\beta_{c_{\Psi}}$ is in fact independent of the choice of
representative element $h \in \psi$, and thus we obtain a
distinguished ``central'' summand $G_0(G,S_\Psi)_{c_\Psi}$ in
$G_0(G,S_\Psi)$.

\begin{theorem}[Non-abelian localization theorem] \cite[Theorem 5.1]{EdGr:05}
\label{thm.nonabelian} 
The pullback map $f^*_\Psi \colon G_0(G,X) \to G_0(G,S_\Psi)$ induces
an isomorphism between the localization of $G_0(G,X)$ at the maximal
ideal $m_\Psi \in \Spec R(G) \otimes \CC$ corresponding to the
conjugacy class $\Psi$ and the summand $G_0(G,S_\Psi)_{c_\Psi}$ in
$G_0(G,S_\Psi)$. Moreover, the Euler class of the normal bundle,
$\lambda_{-1}(N_{f_\Psi^*})$ is invertible in
$G_0(G,S_\Psi)_{c_\Psi}$ and if $\alpha \in G_0(G,X)_{m_\Psi}$ then
\begin{equation}
\alpha = f_{\Psi *}\left({f^*\alpha_{c_\Psi}\over{\lambda_{-1}(N_f^*)}}\right).
\end{equation}
\end{theorem} 
The theorem can be restated in way that is sometimes more useful for
calculations. Fix an element $h \in \Psi$ 
and again let $Z = {\mathcal Z}_G(h)$ be the centralizer of $h$ in $G$.
Let $\iota^! : G_0(G,X) \to G_0(Z,X^h)$
be the composition of the restriction of groups map $G_0(G,X) \to G_0(Z,X)$ with
the pullback $G_0(Z,X) \stackrel{i_h^*} \to
G_0(Z,X^h)$. Let $\beta_h$ denote the component of
$\beta \in G_0(Z,X^h)$ in the summand $G_0(Z,X^h)_{m_h}$. Let ${\mathfrak g}$ (resp. ${\mathfrak z}$) be the
adjoint representation of $G$ (resp. $Z$). The restriction of the adjoint
representation to the subgroup $Z$ makes ${\mathfrak g}$ a $Z$-module, so 
$\mathfrak{g/z}$ is a
$Z$-module. Since $S_\Psi = G\times_Z X^h$, under the identification 
$G_0(G,S_\Psi) = G_0(Z,X^h)$ the class of the 
conormal
bundle of the map $f_\Psi$ is identified with $N_{i_h}^* - \mathfrak{g/z}^*$.
Thus we can restate the non-abelian localization theorem as follows:
\begin{corollary} \label{cor.nonabeliansinglement} Let $\iota _!$ be
  the composite of $f_{\Psi *}$ with the isomorphism $
  G_0(Z,X^h) \to G_0(G,S_\Psi)$.  Then for $\alpha \in G_0(G,X)_{m_\Psi}$
\begin{equation}
\alpha = \iota_!\left(
{(\iota^!\alpha)_{h} \cdot  \lambda_{-1}(\mathfrak {g/z}^*)}\over{
\lambda_{-1}(N^*_{i_h})}\right).
\end{equation}
\end{corollary}
The element $h \in Z(h)$ is central, and as in the abelian case
we obtain a twisting map $t_h \colon G_0(Z,X^h) \to G_0(Z,X^h)$
which maps the summand $G_0(Z,X^h)_h$ to the summand $G_0(Z,X^h)_1$ and also preserves invariants.

We can then obtain the Riemann-Roch theorem in the general case.
Let ${1_G}=\Psi_1, \ldots , \Psi_n$ be conjugacy classes corresponding to 
the support of $G_0(G,X)$ as an $R(G)$ module. Choose a representative element
$h_r \in \Psi_r$ for each $r$. Let $Z_r$ be the centralizer of $h$ in $G$
and let ${\mathfrak z}_r$ be its Lie algebra.
\begin{theorem} \label{thm.rrwithz}
Let $\ix = [X/G]$ be a smooth, complete Deligne-Mumford quotient stack.
Then for any vector bundle $V$ on $\ix$
\begin{equation}
\euler(\ix, V) = \sum_{r = 1}^n  \int_{[X^{h_r}/Z_r)]}\ch \left(
t_{h_r}(\frac{[i_r^*V] \cdot \lambda_{-1}({\mathfrak g^*}/{\mathfrak{z}_r^*})}
{\lambda_{-1}(N^*_{i_{r}})})\right)
\Td([X^{h_r}/Z_r])
\end{equation}
where $i_{r} \colon X^{h_r} \to X$ is the inclusion map.
\end{theorem}

\subsubsection{A computation using Theorem \ref{thm.rrwithz}: The quotient
stack $[(\Pro^2)^3/S_3]$} \label{sec.p3s3}
We now generalize the calculation of Section \ref{sec.p3z3} to
the quotient $\iy =
[Y/S_3]$ where the symmetric group $S_3$ acts on $Y = (\Pro^2)^3$ by 
permutation.
Again we will compute $\euler(\iy, L)$ where $L = {\mathcal O}(m) \boxtimes
{\mathcal O}(m) \boxtimes {\mathcal O}(m)$ viewed as an $S_3$-equivariant
line bundle on $(\Pro^2)^3$. As was the case for the $\ZZ_3$ action
the $S_3$-equivariant rational Chow group is generated by symmetric polynomials
in $H_1,H_2,H_3$ where $H_i$ is the hyperplane class on the $i$-th copy of $\Pro^2$.
The calculation of $\euler(\iy, L_{1})$ is identical to the one we did 
for the stack $\ix = [(\Pro^2)/\ZZ_3]$ except that the cycle $[P \times P \times P]$ has stabilizer $S_3$ which has order 6. Thus,
\begin{equation}
  \euler(\ix, L_1) = 1/6\left(1 + 9 m/2 + 33 m^2/4 + 63 m^3/8 + 33 m^4/8 + 9 m^5)/8 + m^6/8\right).
\end{equation}
Now $\Spec R(S_3) \otimes \CC$ consists of 3 points, corresponding to
the conjugacy classes of $\{1\}$, $\Psi_2 = \{(12), (13), (23)\}$ and
$\Psi_3 = \{(123), (132) \}$.  We denote the components of $L$ at the
maximal ideal corresponding to $\Psi_2$ and $\Psi_3$ by $L_2$ and
$L_{3}$ respectively, so that $L = L_{1} + L_{2} + L_{3}$.

The computation of  $\euler(\iy,L_{3})$ is
identical to the computation of $\euler(\ix, L_{\omega})$ in Section
\ref{sec.p3z3}.  If we choose the representative element $\omega =
(123)$ in $\Psi_3$ then ${\mathcal Z}_{S_3}(\omega) = \langle \omega
\rangle = \ZZ_3$. Again $Y^\omega = \Delta_{(\Pro^2)^3}$ and the
tangent bundle to $(\Pro^2)^3$ restricts to the $\ZZ_3$-equivariant
bundle $T\Pro^2 \otimes V$ where $V$ is the regular
representation. Hence (see \eqref{eq.z3fixed})
\begin{equation}
\euler(\iy, L_{3}) = 1/3(1 + 3 m/2 + m^2/2)
\end{equation}

To compute $\euler(\iy, L_{2})$ choose the representative element $
\tau = (12)$ in the conjugacy class $\Psi = (12)$. Then ${\mathcal Z}_{S_3}
(\tau) =
\langle \tau \rangle = \ZZ_2$ and the fixed locus of $\tau$ is
$Y^{\tau} = \Delta_{(\Pro^2)^2} \times \Pro^2 \stackrel{\Delta_{12}}
\hookrightarrow (\Pro^2)^3$ where $\Delta_{(\Pro^2)^2} \subset (\Pro^2)^2$ is the
diagonal. The action of $\ZZ_2$ is trivial and the tangent bundle
to $(\Pro^2)^{3}$
restricts to $(T_{\Pro^2} \otimes V) \boxtimes T\Pro^2$ where $V$ is
now the regular representation of $\ZZ_2$ so $N_{\Delta_{12}} =
(T\Pro^2 \otimes \xi) \boxtimes T\Pro^2$ where $\xi$ is the
non-trivial character of $\ZZ_2$. Since $\xi$ is self-dual as a
character of $\ZZ_2$ we see that
\begin{equation}
\lambda_{-1}(N_{\Delta_{12}}^*)  =  (1  - (T^*\Pro^2 \otimes \xi) + K_{\Pro^2})
\end{equation}
Applying the twisting operator yields
\begin{equation}
t(\lambda_{-1}(N_{\Delta_{12}}^*) = (1 + T^*\Pro^2\otimes \xi  + K_{\Pro^2}) 
\end{equation}
Taking the Chern character we have 
$$\ch(t(\lambda_{-1}(N_{\Delta_{12}}^*))) = 4 - 6 H + 6 H^2$$
where $H$ is the hyperplane class on the diagonal $\Pro^2$.  The
restriction of $L$ to $Y^{\tau}$ is the line bundle $({\mathcal O}(2m)
\otimes {\mathbf 1}) \boxtimes {\mathcal O}(m)$. Thus,
\begin{eqnarray*}
  \euler(\ix, L_{2}) & = & \int_{[X^{\tau}/\ZZ_2]} \ch({\mathcal O}(2m) \boxtimes
{\mathcal O(m)} \ch(t(\lambda_{-1}(N_{\Delta}^*)^{-1} \Td(Y^\tau)\\
  & = & 1/2\left(\text{ coefficient of }H^2 H_3^2\right)\\
  & = & 1/2(1 + 3 m + 13 m^2/4 + 3 m^3/2 + m^4/4)
\end{eqnarray*}
Adding the Euler characteristics of $L_1, L_2,L_3$ gives
$$\euler(\iy, L) = 1 + 11 m/4 + 19 m^2/6 + 33 m^3/16 + 13 m^4/16 + 
 3 m^5/16 + m^6/48$$
which is again an integer-valued polynomial in $m$.
\subsubsection{Statement of the theorem in terms of the inertia stack}
The computation of $\euler(\ix, \alpha)$ does not depend on the choice of
the representatives of elements in the conjugacy classes and 
Theorem \ref{thm.rrwithz} can be restated in terms of the $S_\Psi$ and correspondingly in terms of the inertia stack.

\begin{definition}
  Let $IX = \{(g,x) | gx = x\} \subset G \times X$ be the inertia
  scheme.  The projection $IX \to X$ makes $IX$ into a group scheme
  over $X$. If the stack $[X/G]$ is separated then $IX$ is finite over
$X$.

The group $G$ acts on $IX$ by $g(h,x) = (ghg^{-1}, gx)$ and the projection
$IX \to X$ is $G$-equivariant with respect to this action. The quotient
stack $I\ix :=[IX/G]$ is called the {\em inertia stack} of the stack
$\ix = [X/G]$ and there is an induced morphism of stacks $I\ix \to \ix$.
Since $G$ acts properly on $X$ then the map $I \ix \to \ix$ is finite and unramified.
\end{definition}

Since $G$ acts with finite stabilizers a necessary condition for 
$(g,x)$ to be in $IX$ is for $g$ to be of finite order.
\begin{proposition}
  If $\Psi$ is a conjugacy class of finite order then $S_\Psi$ is
  closed and open in $IX$ and consequently there is a finite $G$-equivariant
  decomposition $IX = \coprod_{\Psi} S_\Psi$.
\end{proposition}
Since $IX$ has a $G$-equivariant decomposition into a finite disjoint
sum of the $S_\Psi$ we can define a twisting automorphism $t \colon
G_0(G,IX)\otimes \CC \to G_0(G,IX)\otimes \CC$ and thus a corresponding
twisting action on $G_0(I\ix)$. If $V$ is a $G$-equivariant vector
bundle on $IX$ then its fiber at a point $(h,x)$ is ${\mathcal
  Z}_G(h)$-module $V_{h,x}$ and $t(V)$ is the class in $G_0(G,IX)\otimes
\CC$ whose ``fiber'' at the point $(h,x)$ is the virtual ${\mathcal
  Z}_G(h)$-module $\oplus_{\xi \in X(H)} \xi(h) (V_{h,x})_{\xi}$ where
$H$ is the cyclic group generated by $h$.

The Hirzebruch-Riemann-Roch theorem can then be stated very concisely as: 
\begin{theorem}
Let $\ix = [X/G]$ be a smooth, complete Deligne-Mumford quotient stack
and let $f \colon I\ix \to \ix$ be the inertia map.
If $V$ is a vector bundle on $\ix$ then
$$\euler(\ix, V) = \int_{I\ix} \ch\left(t({f^*V\over{\lambda_{-1}(N_f^*)}})\right)\Td(I\ix)$$
\end{theorem}

\section{Grothendieck Riemann-Roch for proper morphisms of Deligne-Mumford  quotient stacks} \label{sec.dmgrr}
In the final section we state the
Grothendieck-Riemann-Roch theorem for arbitrary proper morphisms of 
quotient Deligne-Mumford stacks.

\subsection{Grothendieck-Riemann-Roch for proper morphisms to schemes and algebraic spaces}
The techniques used to prove the Hirzebruch-Riemann-Roch for proper Deligne-Mumford stacks actually yield a Grothendieck-Riemann-Roch 
result for arbitrary separated Deligne-Mumford stacks relative to map $\ix \to M$ where $M$ is the moduli space of the quotient stack $\ix = [X/G]$.

\begin{theorem} \label{thm.rrforquotients} \cite[Theorem 6.8]{EdGr:05}
Let $\ix = [X/G]$ be a smooth quotient stack with coarse moduli space
$p \colon \ix \to M$. Then the following diagram commutes:
$$\begin{array}{ccc}
G_0(\ix) & \stackrel{I\tau_{\ix}} \to & \Chow^*(I\ix) \otimes \CC\\
p_* \downarrow & & p_* \downarrow \\
G_0(M) & \stackrel{\tau_M} \to & \Chow^*(M) \otimes \CC
\end{array}.
$$
Here $I\tau_{\ix}$ is the isomorphism that sends the
class in $G_0(\ix)$ of a vector bundle $V$ to
$\ch\left(t({f^*V\over{\lambda_{-1}(N_f^*)}})\right)\Td(I\ix)$ and
$\tau_M$ is the Fulton-MacPherson Riemann-Roch isomorphism.
\end{theorem}
\begin{remark}
If $\ix$ is satisfies the resolution property then every coherent sheaf on $\ix$
can be expressed as a linear combination of classes of vector bundles.
\end{remark}
Using the universal property of the coarse moduli space and the
covariance of the Riemann-Roch map for schemes and algebraic spaces we
obtain the following Corollary.
\begin{corollary}
  Let $\ix = [X/G]$ be a smooth quotient stack and let $\ix \to Z$ be
  a proper morphism to a scheme or algebraic space.  Then the following diagram commutes:
$$\begin{array}{ccc}
G_0(\ix) & \stackrel{I\tau_{\ix}} \to & \Chow^*(I\ix) \otimes \CC\\
p_* \downarrow & & p_* \downarrow \\
G_0(Z) & \stackrel{\tau_Z} \to & \Chow^*(Z) \otimes \CC.
\end{array}
$$
\end{corollary}

\subsubsection{Example: The Todd class of a weighted projective space}

If $X$ is an arbitrary scheme we define the {\em Todd
  class}, $\td(X)$, of $X$ to be $\tau_X({\mathcal O}_X)$ where $\tau_X$ is
the Riemann-Roch map of Theorem \ref{thm.fultrr}.  If $X$ is smooth,
then $\td(X) = \Td(TX)$, and for arbitrary complete schemes
$\euler(X,V) = \int_X \ch(V) \td(X)$ for any vector bundle $V$ on $X$.

In this section we explain how to use Theorem \ref{thm.rrforquotients} to
give a formula for the Todd class of the singular weighted projective
space ${\mathbf P}(1,1,2)$. The method can be extended to any simplicial
toric variety, complete or not, \cite{EdGr:03}. (See also \cite{BrVe:97}
for a computation of the equivariant Todd class of complete toric varieties using other methods.)

The singular variety ${\bf P}(1,1,2)$ is the quotient of $X=\A^3 
\smallsetminus \{0\}$ where
$\CC^*$ acts with weights $(1,1,2)$. This variety is the coarse moduli space of the corresponding smooth stack $\Pro(1,1,2)$. A calculation similar to that of 
Section \ref{sec.supp} shows that  $K_0(\Pro(1,1,2)) = 
\ZZ[\xi]/(\xi-1)^2(\xi^2-1)$ and $\Chow^*(\Pro(1,1,2)) = \ZZ[t]/2t^3$ where $t=c_1(\xi)$. 

The stack $\Pro(1,1,2)$ is a toric Deligne-Mumford stack (in the sense
of \cite{BCS:05}) and the weighted projective space ${\bf P}(1,1,2)$
is the toric variety $X(\Sigma)$ where $\Sigma$ is the complete
2-dimensional fan with rays by $\rho_0 = (-1,-2), \rho_1 = (1,0),
\rho_2 = (0,1)$. This toric variety has an isolated singular point
$P_0$ corresponding to the cone $\sigma_{01}$ spanned by $\rho_0$ and
$\rho_1$.
$$
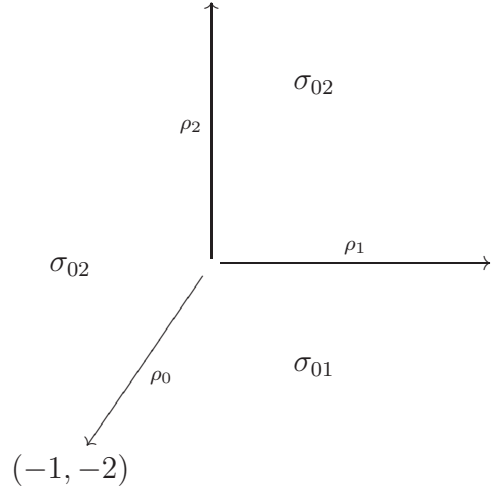
\xymatrix{
 & &  & & & \\
& & \sigma_{02} & & & \\
&  &  &  & & \\
\sigma_{02} &  \ar[ddl]^{\rho_0} \ar[rrr]^{\rho_1} \ar[uuu]^{\rho_2} & & & &\\
& & \sigma_{01} & & & \\
(-1,-2)& & & & &
}
$$

Each ray determines a Weil divisor $D_{\rho_i}$ which is the image of
the fundamental class of the hyperplane $x_i =0$. With the given
action, $[x_0=0] = [x_1 =0] = t$ and $[x_2=0] = 2t$. Since the action of $\CC^*$
on $\A^3$ is free on the complement of a set of codimension
2, the pushforward defines an isomorphism of integral Chow groups
$\Chow^1(\Pro(1,1,2)) =
\Chow^1({\mathbf P}(1,1,2))$. Thus, $\Chow^1(X(\Sigma) = \ZZ$ and $D_{\rho_0}
\equiv D_{\rho_1}$ while $D_{\rho_2} \equiv 2 D_{\rho_0}$.  Also,
$\Chow^2(X(\Sigma)) = \ZZ$ is generated by the class of the singular point $P_0$ 
and 
$[P_0] = 2[P]$ for any non-singular point $P$. 

The tangent bundle to $\Pro(1,1,2)$ fits into the Euler sequence
$$0 \to {\mathbf 1} \to  2 \xi + \xi^2 \to T\Pro(1,1,2) \to 0$$
so $c_1(T\Pro(1,1,2) = 4t$
and $c_2(T\Pro(1,1,2) = 15t^2$. Thus 
$$\Td(\Pro(1,1,2)) = 1 +2t +  21/12 t^2.$$
Pushing forward to ${\mathbf P}(1,1,2)$ gives a contribution
of $1 + 2D_{\rho_0} + 21/24 P_0$ to $\td({\mathbf P}(1,1,2))$.

Now we must also consider the contribution coming from the fixed locus
of $(-1)$ acting on $\A^3 \smallsetminus \{0\}$. The fixed locus is the line
$x_0 = x_1=0$ and the normal bundle has $K$-theory class  $2\xi$.
After twisting by $-1$  we obtain a contribution of
\begin{equation} \label{eq.minuscontrib}
p_*\left[\ch\left({1\over{(1 + \xi^{-1})^2}}\right)\Td([X^{-1}/\CC^*])\right]
\end{equation}

Since $[X^{-1}/\CC^*]$
is 0-dimensional and has a generic stabilizer of order 2
we obtain an additional contribution of $1/2 \rk( 1/(1 + \xi^{-1})^2)[P_0]
= (1/2\times 1/4)[P_0]  = 1/8[P_0]$ to $\td({\mathbf P}(1,1,2))$.
Combining the two contributions we conclude that:
$$\td({\mathbf P}(1,1,2)) = 1 + 2D_{\rho_0} + [P_0]$$
in $\Chow^*({\bf P}(1,1,2)$.

\subsection{Grothendieck-Riemann-Roch for Deligne-Mumford quotient
  stacks}

Suppose that $\ix = [X/G]$ and $\iy = [Y/H]$ are smooth
Deligne-Mumford quotient stacks and $f \colon \ix \to \iy$ is a
proper, but not-necessarily representable morphism. The most general
Grothendieck-Riemann-Roch result we can write down is the following:
\begin{theorem} \label{thm.dmgrr}\cite{EdKr:12}
The  following diagram of Grothendieck groups and Chow groups commutes:
$$\begin{array}{ccc}
G_0(\ix) & \stackrel{I\tau_{\ix}} \to & \Chow^*(I\ix) \otimes \CC\\
f_* \downarrow & & f_* \downarrow \\
G_0(\iy) & \stackrel{I\tau_{\iy}} \to & \Chow^*(I\iy) \otimes \CC
\end{array}
$$
\end{theorem}
\begin{remark}
  A proof of this result using the localization methods of
  \cite{EdGr:03, EdGr:05} will appear in \cite{EdKr:12}.  A version of
  this Theorem (which also holds in some prime characteristics) was
  proved by Bertrand Toen in \cite{Toe:99}.  However, in that paper
  the target of the Riemann-Roch map is not the Chow groups but rather
  a ``cohomology with coefficients in representations.''  Toen does
  not explicitly work with quotient stacks, but his hypothesis that
  the stack has the resolution property for coherent sheaves implies
  that the stack is a quotient stack.

In \cite{EdKr:12} we will also give a version of Grothendieck-Riemann-Roch for proper morphisms of arbitrary quotient stacks.
\end{remark}

\section{Appendix on $K$-theory and Chow groups}
In this section we recall some basic facts about $K$-theory and Chow groups both in the non-equivariant and equivariant settings.
For more detailed references see \cite{Ful:84, FuLa:85, Tho:87a, EdGr:98}.

\subsection{$K$-theory of schemes and algebraic spaces}

\begin{definition}
Let $X$ be an algebraic scheme. We denote by $G_0(X)$ the Grothendieck group
of coherent sheaves on $X$ and $K_0(X)$ the Grothendieck group of locally free sheaves; i.e vector bundles.
\end{definition}
There is a natural map $K_0(X) \to G_0(X)$ which is an isomorphism when
$X$ is a smooth scheme. The reason is that if $X$ is smooth every
coherent sheaf has a finite resolution by locally free sheaves. For a
proof see \cite[Appendix B8.3]{Ful:84}.

\begin{definition}
If $X \to Y$ is a proper morphism then there is a pushforward map
$f_* \colon G_0(X)\to G_0(Y)$ defined by 
$f_*[{\mathcal F}] = \sum_i (-1)^i [R^if_*{\mathcal F}]$. When $Y = \pt$,
then $G_0(Y) = \ZZ$ and $f_*({\mathcal F}) = \chi(X,{\mathcal F})$.
\end{definition}

The Grothendieck group $K_0(X)$ is a ring under tensor product and the
map $K_0(X) \otimes G_0(X) \to G_0(X)$, $([V],{\mathcal F}) \mapsto {\mathcal F} \otimes V$ makes $G_0(X)$ into a $K_0(X)$-module.
If $f \colon X \to Y$ is an arbitrary morphism of schemes then pullback 
of vector bundles defines a ring homomorphism
$f^* \colon K_0(Y) \to K_0(X)$.

When $f \colon X \to Y$ is proper, the pullback for vector bundles and
the pushforward for coherent sheaves are related by the projection
formula. Precisely, if $\alpha \in K_0(Y)$ and $\beta \in G_0(X)$ then
$$f_*(f^*\alpha \cdot \beta) = \alpha \cdot f_*\beta$$
in $G_0(Y)$.

There is large class of morphisms $X \stackrel{f} \to Y$, 
for which there are pullbacks $f^* \colon G_0(Y) \to G_0(X)$ and pushforwards
$f_* \colon K_0(X) \to K_0(Y)$. For example, if $f$ is flat, the assignment
$[{\mathcal F}] \mapsto [f^*{\mathcal F}]$ defines a pullback $f^* \colon 
G_0(Y) \to G_0(X)$.

Suppose that every coherent sheaf on $Y$ is the quotient
of a locally free sheaf (for example if $Y$
embeds into a smooth scheme). If $f \colon X \to Y$ is a regular embedding 
then the direct
image $f_*V$ of a locally free sheaf has a finite resolution $W_{.}$ by locally free
sheaves. Thus we may define a pushforward $f_* \colon K_0(X) \to K_0(Y)$ by
$f_*[V] = \sum_{i} (-1)^i [W_i]$ in this case. Also, if $X$ and $Y$ are smooth 
then there is a pushforward $f_* \colon K_0(X) \to K_0(Y)$. When
$X$ and $Y$ admit ample line bundles then there are pushforwards
$f_* \colon K_0(X) \to K_0(Y)$ for any proper morphism of finite Tor-dimension.

\begin{definition}
The Grothendieck ring $K_0(X)$ has an additional structure as a
$\lambda$-ring.  If $V$ is a vector bundle of rank $r$ set
$\lambda^k[V] = [\Lambda^kV]$. If $t$ is parameter define
$\lambda_t(V) = \sum_{k=0}^r \lambda^k[V] t^k \in K_0(X)[t]$ where $t$
is a parameter.  
The class $\lambda_{-1}(V^*) = 1 - [V^*] + [\Lambda^2 V^*]
+ \ldots + (-1)^r[\Lambda^rV^*]$ is called the {\em $K$-theoretic Euler
  class} of $V$.
\end{definition}

Although, $K_0(X)$ is simpler to define and is functorial for arbitrary
morphisms, it is actually much easier to prove results about the
Grothendieck group $G_0(X)$. The reason is that $G$-functor behaves
well with respect to localization. If $U \subset X$ is open with
complement $Z$ then there is an exact sequence
$$ G_0(Z) \to G_0(X) \to G_0(U) \to 0.$$

The definitions of $G_0(X)$ and $K_0(X)$ also extend to algebraic spaces
as does the basic functoriality of these groups. However, even if $X$
is a smooth algebraic space there is no result guaranteeing that $X$ satisfies the {\em
  resolution property} meaning that every coherent sheaf is the
quotient of a locally free sheaf.  Thus it is not possible to prove
that the natural map $K_0(X) \to G_0(X)$ is actually an isomorphism. (Note
however, that there no known examples of smooth separated algebraic
spaces where the resolution property provably fails, c.f. \cite{Tot:04}.)
In this case one can either replace $K_0(X)$ with the Grothendieck group
of perfect complexes or work exclusively with $G_0(X)$. 

\subsection{Chow groups of schemes and algebraic spaces}

\begin{definition}
  If $X$ is a scheme (which for simplicity we assume to be equi-dimensional)
we denote by $\Chow^i(X)$ the Chow group of
codimension  $i$-dimensional cycles modulo rational equivalence as in
  \cite{Ful:84} and we set $\Chow^*(X) =  \oplus_{i=0}^{\dim X} \Chow^i(X)$ . 
\end{definition}

As was the case for the Grothendieck group $G_0(X)$, if $f \colon X \to Y$
is proper then there is a pushforward
$f_* \colon \Chow^*(X)  \to \Chow^*(Y)$. The map is defined as follows:
\begin{definition}
If $Z \subset X$ is a closed subvariety let $W= f(Z)$ with its reduced scheme
structure
$$f_*[Z] = \left\{ \begin{array}{cc} [K(Z):K(W)] [W] & \text{if } 
\dim W = \dim Z\\      0 & \text{otherwise} \end{array}\right\}$$
where $K(Z)$ (resp. $K(W)$) is the function field of $Z$ (resp. $W$).

If $X$ is complete then we denote the pushforward map $\Chow^*X \to \Chow^*(\pt) = \ZZ$ by $\int_X$.

\end{definition}
Because we index our Chow groups by codimension, the map $f_*$ shifts
degrees. If $f \colon X \to Y$ has (pure) relative dimension $d$
then $f_*(\Chow^k(X)) \subset \Chow^{k+d}(Y)$.

There is again a large class of morphisms $X \stackrel{f} \to Y$ for
which there are pullbacks $f^* \colon \Chow^*(Y) \to \Chow^*(X)$. Some
of the most important examples are flat morphisms where the pullback
is defined by $f^*[Z] = [f^{-1}(Z)]$, regular embeddings and, more
generally, local complete intersection morphisms.

We again have a localization exact sequence which can be used for computation. 
If $U \subset X$ is open
with complement $Z$ then there is a short exact sequence
$$ \Chow^*(Z) \to \Chow^*(X) \to \Chow^*(U) \to 0$$

\begin{definition}
If $X$ is smooth (and separated) then the diagonal $ \Delta \colon X \to X \times X$ is a regular embedding. Pullback along the diagonal allows us to define
an intersection product on $\Chow^*(X)$ making it into a graded ring, called
the Chow ring.
If $[Z] \subset \Chow^k(X)$ and $[W] \subset \Chow^l(X)$ then 
we define $[Z] \cdot [W] = \Delta^*([Z\times W]) \in \Chow^{k+l}(X)$.
\end{definition}

Any morphism of smooth varieties is a local complete intersection
morphism, so if $f \colon X \to Y$ is a morphism of smooth varieties
then we have a pullback $f^* \colon \Chow^*Y \to \Chow^*X$ which is a
homomorphism of Chow rings.

The theory of Chow groups carries through completely to algebraic spaces
\cite[Section 6.1]{EdGr:98}.
\subsection{Chern classes and operations}
Associated to any vector bundle $V$ on a scheme $X$ are Chern classes
$c_i(V)$, $0 \leq i \leq \rk V$. Chern classes are defined as
operations on Chow groups. Specifically $c_i(V)$ defines a homomorphism 
$\Chow^kX \to \Chow^{k+i}X$, $\alpha \mapsto c_i(V)\alpha$, with $c_0$ 
taken to be the identity map and denoted by $1$.
Chern classes
are compatible with pullback in the following sense:
If $f \colon X \to Y$ is a morphism for which there is a pullback of
Chow groups then $c_i(f^*V)f^*\alpha = f^*(c_i(V) \alpha)$. 

Chern classes of a vector bundle $V$ may be viewed as elements of the
{\em operational Chow ring} $A^*X = \oplus_{i = 0}A^iX$ defined in
\cite[Definition 17.3]{Ful:84}. An element of $c \in A^iX$ is a
collection of homomorphisms $c \colon \Chow^*(X') \to \Chow^{*+k}(X')$
defined for any morphism of schemes $X' \to X$. These homomorphisms
should be compatible with pullbacks of
Chow groups and should also satisfy the projection formula $f_*(c \alpha) = c
f_*\alpha$ for any proper morphism of $X$-schemes $f\colon X'' \to X'$
and class $\alpha \in \Chow^*(X'')$. Composition of morphisms makes
$A^*X$ into a graded ring and it can be shown that $A^kX = 0$ for $k >
\dim X$.

If $X$ is smooth, then the map $A^*X \to \Chow^*X$, $c \mapsto c([X])$
is an isomorphism of rings where the product on $\Chow^*X$ is the
intersection product.  In particular, if $X$ is smooth then the Chern
class $c_i(V)$ is completely determined by $c_i(V)[X] \in \Chow^i(X)$
so in this way we may view $c_i(V)$ as an element of $\Chow^i(X)$.

The total Chern class $c(V)$ of a vector bundle is the sum
$\sum_{i=0}^{\rk V} c_i(V)$. Since $c_0 =1$ and $c_i(V)$ is nilpotent
for $i> 0$ the total Chern class $c(V)$ is invertible in $A^*X$. Also,
if
$0 \to V' \to V \to V'' \to 0$ is a short exact sequence of vector bundles
then $c(V) = c(V')c(V'')$, so the assignment
$[V] \mapsto c(V)$ defines a homomorphism from the Grothendieck group
$K_0(X)$ to the multiplicative group of units in $A^*X$.

\subsubsection{Splitting, Chern characters and Todd classes}
If $V$ is a vector bundle on a scheme $X$, then the {\em splitting
  construction} ensures that there is a scheme $X'$ and a smooth,
proper morphism $f \colon X' \to X$ such that $f^* \colon \Chow^*X \to
\Chow^*X'$ is injective and $f^*V$ has a filtration $0 = E_0 \subset
E_1 \subset \ldots E_r=f^*V$ such that the quotients $L_i =
E_i/E_{i-i}$ are line bundles. Thus $c(f^*V)$ factors as
$\prod_{i=1}^r (1 + c_1(L_i))$. The classes $\alpha_i = c_1(L_i)$ are
Chern roots of $V$ and any symmetric expression in the $\alpha_i$ is the
pullback from $\Chow^*X$ of a unique  expression in the Chern classes of $V$.

\begin{definition}
If $V$ is a vector bundle on $X$ with Chern roots $\alpha_1, \ldots \alpha_r
\in A^*X'$ for some $X' \to X$ then the {\em Chern character} of $V$
is the unique class $\ch(V)\in A^*X \otimes \QQ$ which pulls back to 
$\sum_{i=0}^r \exp(\alpha_i)$ in $A^*(X')\otimes \QQ$. (Here
$\exp$ is the exponential series.)

Likewise the {\em Todd class} of $V$ is the unique class $\Td(V) \in A^*X \otimes \QQ$
which pulls back to
$\prod_{i=0}^r {\alpha_i\over{1- \exp(-\alpha_i)}}$ in $A^*(X') \otimes \QQ$.

The Chern character can be expressed in terms of the Chern classes of $V$ as
\begin{equation} \label{eq.chernchar}
\ch(V) = \rk V + c_1 + (c_1^2 -c_2)/2 + \ldots 
\end{equation}
and the Todd class as
\begin{equation} \label{eq.toddclass}
\Td(V) = 1 + c_1/2 + (c_1^2 + c_2)/12 + \ldots 
\end{equation}
Because $A^k(X) = 0$ for $k > \dim X$ the series for $\ch(V)$ and
$\Td(X)$ terminate for any given scheme $X$ and vector bundle
$V$.
\end{definition}

If $V$ and $W$ are vector bundles on $X$ then $\ch(V\oplus W) = \ch(V) + \ch(W)$
and $\ch(V \otimes W) = \ch(V) \ch(W)$ so the Chern character defines a
homomorphism of rings
$\ch\colon K_0(X) \to A^*X \otimes \QQ.$
We also have that $\Td(V \oplus W) = \Td(V)\Td(W)$
so we obtain a homomorphism
$\Td \colon K_0(X) \to (A^*X \otimes \QQ)^{\star}$ from the additive 
Grothendieck group to the multiplicative group of units in $A^*X \otimes \QQ$.

When $X$ is smooth we interpret the target
of the Chern character and Todd class to be $\Chow^*X$.

\subsection{Equivariant $K$-theory and equivariant Chow groups}
We now turn to the equivariant analogues of Grothendieck and Chow groups.
\subsubsection{Equivariant $K$-theory}
Most of the material on equivariant $K$-theory can be found in \cite{Tho:87a}
while the material on equivariant Chow groups is in \cite{EdGr:98}.

\begin{definition}
Let $X$ be a scheme (or algebraic space) with the action of an algebraic group
$G$. In this case we define $K_0(G,X)$ to be the Grothendieck group of $G$-equivariant vector bundles and $G_0(G,X)$ to be the Grothendieck group of $G$-linearized coherent sheaves.
\end{definition}

As in the non-equivariant case there is pushforward of Grothendieck 
groups $G_0(G,X) \to G_0(G,Y)$ for any proper $G$-equivariant morphism. Similarly,
there is a pullback $K_0(G,Y) \to K_0(G,X)$ for any $G$-equivariant morphism
$X \to Y$. There are also pullbacks in $G$-theory for equivariant
regular embeddings and equivariant lci morphisms. There is also
a localization exact sequence associated to a $G$-invariant open set
$U$ with complement $Z$. 

The Grothendieck group $K_0(G,X)$ is a ring under tensor product and
$G_0(G,X)$ is a module for this ring. The equivariant Grothendieck ring
$K_0(G,\pt)$ is the representation ring $R(G)$ of $G$. Since every scheme maps
to a point, $R(G)$ acts on both $G_0(G,X)$ and $K_0(G,X)$ for any $G$-scheme $X$.
The $R(G)$-module structure on $G_0(G,X)$ plays a crucial role in the
Riemann-Roch theorem for Deligne-Mumford stacks.

If $V$ is a $G$-equivariant vector bundle then $\Lambda^kV$ has
a natural $G$-equivariant structure. This means that the wedge product defines
a  $\lambda$-ring structure on $K_0(G,X)$. In particular
we define the equivariant Euler class of a rank $r$  bundle
$V$ by the formula
$$\lambda_{-1}(V^*) = 1 - [V^*] + [\Lambda^2V^*] - \ldots + (-1)^r [\Lambda^rV^*].$$

Results of Thomason \cite[Lemmas 2.6, 2.10, 2.14]{Tho:87} imply that if $X$ is normal 
and quasi-projective or regular and separated
over the ground field (both of which implies that $X$ has the resolution property)
and $G$ acts on $X$
then $X$ has the $G$-equivariant resolution property. It follows that if
$X$ is a smooth $G$-variety then every $G$-linearized coherent sheaf has a finite resolution by $G$-equivariant vector bundles. Hence $K_0(G,X)$ and $G_0(G,X)$
may be identified if $X$ is a smooth scheme.

The Grothendieck groups $G_0(G,X)$ and $K_0(G,X)$ are naturally identified with
the corresponding Grothendieck groups of the categories of locally free and 
coherent sheaves on the quotient stack $\ix = [X/G]$.

\begin{remark}[Warning]
If $X$ is complete then there are pushforward maps $K_0(G,X) \to K_0(G, \pt)=R(G)$
and $G_0(G,X) \to K_0(G,\pt)= R(G)$ that associate to a vector bundle $V$
(resp. coherent sheaf ${\mathcal F}$) the virtual representation
$\sum (-1)^i H^i(X,V)$ (resp. $\sum (-1)^i H^{i}(X, {\mathcal F})$.). Although
$V$ may be viewed as a vector bundle on the quotient stack $\ix = [X/G]$
the virtual representation $\sum (-1)^i H^i(X,V)$ {\em is not} the 
Euler characteristic of $V$ as a vector bundle on $\ix$.
\end{remark}

\subsubsection{Equivariant Chow groups}
The definition of equivariant Chow groups requires more care and is
modeled on the Borel construction in equivariant cohomology.  If $G$
acts on $X$ then the $i$-th equivariant Chow group is defined as
$\Chow^i(X_G)$ where $X_G$ is any quotient of the form $(X \times
U)/G$ where $U$ is an open set in a representation ${\mathbf V}$ of
$G$ such that $G$ acts freely on $U$ and ${\mathbf V} \smallsetminus
U$ has codimension more than $i$.  In \cite{EdGr:98} it is shown that
such pairs $(U, {\mathbf V})$ exist for any algebraic group and that
the definition of $\Chow^i_G(X)$ is independent of the choice of $U$
and ${\mathbf V}$.

Because equivariant Chow groups are defined as Chow groups of certain schemes,
they enjoy all of the functoriality of ordinary Chow groups. 
In particular, if $X$ is smooth then pullback along the diagonal
defines an intersection product on 
$\Chow^*_G(X)$.

\begin{remark}
Intuitively an equivariant cycle may be viewed as a $G$-invariant cycle
on $X \times {\mathbf V}$ where ${\mathbf V}$ is some representation of
$G$. Because representations can have arbitrarily large dimension $\Chow^i(X)$
can be non-zero for all $i$.

If $G$ acts freely then a quotient $X/G$ exists as an algebraic space
and $\Chow^i_G(X) = \Chow^i(X/G)$.
More generally, if $G$ acts with finite stabilizers then elements of 
$\Chow^i_G(X) \otimes \QQ$ are represented by $G$-invariant cycles on $X$
and consequently $\Chow^i_G(X) = 0$ for $i > \dim X - \dim G$.
\end{remark}

As in the non-equivariant case, an equivariant vector bundle $V$ 
on a $G$-scheme 
defines Chern class operations $c_i(V)$ on $\Chow^*_G(X)$. The Chern
class naturally live in the equivariant operational Chow ring $A^*_G(X)$
and as in the non-equivariant case the map
$A^*_G(X) \to \Chow^*_G(X)$, $c \mapsto c[X]$ is a ring isomorphism if $X$
is smooth.

We can again define the Chern character and Todd class of a vector bundle
$V$. However, because $\Chow^i_G(X)$ can be non-zero for all $i$, 
the target of the Chern character and Todd class is the infinite
direct product $\Pi_{i=0}^\infty \Chow^i_G(X) \otimes \QQ$.

When $G$ acts on $X$ with finite stabilizers then $\Chow^i_G(X) \otimes \QQ$
is $0$ for $i > \dim X - \dim G$ so in this case the target of the Chern 
character and Todd class map is $\Chow^*_G(X)$.

\bibliographystyle{amsmath}
%\bibliography{refs}

\begin{thebibliography}{EHKV}

\bibitem[BCS]{BCS:05}
Lev~A. Borisov, Linda Chen, and Gregory~G. Smith, {\em The orbifold {C}how ring
  of toric {D}eligne-{M}umford stacks}, J. Amer. Math. Soc. \textbf{18} (2005),
  no.~1, 193--215 (electronic).

\bibitem[BV]{BrVe:97}
Michel Brion and Mich{\`e}le Vergne, {\em An equivariant {R}iemann-{R}och
  theorem for complete, simplicial toric varieties}, J. Reine Angew. Math.
  \textbf{482} (1997), 67--92.

\bibitem[EG1]{EdGr:98}
Dan Edidin and William Graham, {\em Equivariant intersection theory}, Invent.
  Math. \textbf{131} (1998), no.~3, 595--634.

\bibitem[EG2]{EdGr:00}
\bysame, {\em {R}iemann-{R}och for equivariant {C}how groups}, Duke Math. J.
  \textbf{102} (2000), no.~3, 567--594.

\bibitem[EG3]{EdGr:03}
\bysame, {\em {R}iemann-{R}och for quotients and {T}odd classes of simplicial
  toric varieties}, Comm. in Alg. \textbf{31} (2003), 3735--3752.

\bibitem[EG4]{EdGr:05}
\bysame, {\em Nonabelian localization in equivariant {$K$}-theory and
  {R}iemann-{R}och for quotients}, Adv. Math. \textbf{198} (2005), no.~2,
  547--582.

\bibitem[EHKV]{EHKV:01}
Dan Edidin, Brendan Hassett, Andrew Kresch, and Angelo Vistoli, {\em Brauer
  groups and quotient stacks}, Amer. J. Math. \textbf{123} (2001), no.~4,
  761--777.

\bibitem[EK]{EdKr:12}
Dan Edidin and Amalendu Krishna, {\em Grothendieck-{R}iemann-{R}och for
  equivariant {$K$}-theory}, In preparation (2012).

\bibitem[Ful]{Ful:84}
William Fulton, {\em Intersection theory}, Springer-Verlag, Berlin, 1984.

\bibitem[FL]{FuLa:85}
William Fulton and Serge Lang, {\em {R}iemann-{R}och algebra}, Grundlehren der
  Mathematischen Wissenschaften [Fundamental Principles of Mathematical
  Sciences], vol. 277, Springer-Verlag, New York, 1985.

\bibitem[Gil]{Gil:84}
Henri Gillet, {\em Intersection theory on algebraic stacks and
  {$Q$}-varieties}, Proceedings of the {L}uminy conference on algebraic
  {$K$}-theory ({L}uminy, 1983), vol.~34, 1984, pp.~193--240.

\bibitem[Kaw]{Kaw:79}
Tetsuro Kawasaki, {\em The {R}iemann-{R}och theorem for complex
  {$V$}-manifolds}, Osaka J. Math. \textbf{16} (1979), no.~1, 151--159.

\bibitem[KM]{KeMo:97}
Se{\'a}n Keel and Shigefumi Mori, {\em Quotients by groupoids}, Ann. of Math.
  (2) \textbf{145} (1997), no.~1, 193--213.

\bibitem[Kre]{Kre:99}
Andrew Kresch, {\em Cycle groups for {A}rtin stacks}, Invent. Math.
  \textbf{138} (1999), no.~3, 495--536.

\bibitem[KV]{KrVi:04}
Andrew Kresch and Angelo Vistoli, {\em On coverings of {D}eligne-{M}umford
  stacks and surjectivity of the {B}rauer map}, Bull. London Math. Soc.
  \textbf{36} (2004), no.~2, 188--192.

\bibitem[MFK]{MFK:94}
D.~Mumford, J.~Fogarty, and F.~Kirwan, {\em Geometric invariant theory}, third
  ed., Springer-Verlag, Berlin, 1994.

\bibitem[Seg]{Seg:68b}
Graeme Segal, {\em Equivariant {$K$}-theory}, Inst. Hautes \'Etudes Sci. Publ.
  Math. (1968), no.~34, 129--151.

\bibitem[Ses]{Ses:72}
C.~S. Seshadri, {\em Quotient spaces modulo reductive algebraic groups}, Ann.
  of Math. (2) \textbf{95} (1972), 511--556; errata, ibid. (2) 96 (1972), 599.

\bibitem[Tho1]{Tho:87a}
R.~W. Thomason, {\em Algebraic {$K$}-theory of group scheme actions}, Algebraic
  topology and algebraic $K$-theory (Princeton, N.J., 1983), Ann. of Math.
  Stud., vol. 113, Princeton Univ. Press, Princeton, NJ, 1987, pp.~539--563.

\bibitem[Tho2]{Tho:87}
\bysame, {\em Equivariant resolution, linearization, and {H}ilbert's fourteenth
  problem over arbitrary base schemes}, Adv. in Math. \textbf{65} (1987),
  no.~1, 16--34.

\bibitem[Tho3]{Tho:92}
\bysame, {\em Une formule de {L}efschetz en {$K$}-th\'eorie \'equivariante
  alg\'ebrique}, Duke Math. J. \textbf{68} (1992), no.~3, 447--462.

\bibitem[Toe]{Toe:99}
B.~Toen, {\em Th\'eor\`emes de {R}iemann-{R}och pour les champs de
  {D}eligne-{M}umford}, $K$-Theory \textbf{18} (1999), no.~1, 33--76.

\bibitem[Tot]{Tot:04}
Burt Totaro, {\em The resolution property for schemes and stacks}, J. Reine
  Angew. Math. \textbf{577} (2004), 1--22.

\bibitem[Vis]{Vis:89}
Angelo Vistoli, {\em Intersection theory on algebraic stacks and on their
  moduli spaces}, Invent. Math. \textbf{97} (1989), no.~3, 613--670.

\end{thebibliography}
\def\cprime{$'$}

\end{document}